\documentclass[a4paper, 9pt]{article}
\usepackage{graphics}
\usepackage{color}

\usepackage{lscape}
\usepackage{latexsym}
\usepackage{amsmath,verbatim}
\usepackage{amsthm}
\usepackage{amssymb}
\usepackage[mathscr]{euscript}
\usepackage{yfonts}
\usepackage{makeidx}
\usepackage{stmaryrd}
\usepackage{multicol}
\usepackage{bm}
\usepackage[all]{xy}
\numberwithin{equation}{section}
\newtheorem{thm}{Theorem}[section]
\newtheorem{prop}[thm]{Proposition}
\newtheorem{lem}[thm]{Lemma}
\newtheorem{cor}[thm]{Corollary}
{\bf}{\it}

\newtheorem{fthm}{Theorem}{\bf}{\it}
{\bf}{\it}
\newtheorem{fcor}[fthm]{Corollary}{\bf}{\it}
{\bf}{\it}
{\bf}{\it}

\theoremstyle{definition}

{\bf}{\rm}

\theoremstyle{remark}
\newtheorem{ex}[thm]{Example}
\newtheorem{rem}[thm]{Remark}
{\bf}{\it}

\newtheorem{definition and corollary}[thm]{Definition and Corollary}

{\it}{\rm}

\newcommand{\al}{\alpha}

\newcommand{\C}{{\mathbb C}}

\newcommand{\Par}{{\mathcal P}}

\newcommand{\Hom}{\mathrm{Hom}}

\newcommand{\ch}{\mathrm{ch}}

\newcommand{\gch}{\mathrm{gch}}
\newcommand{\gdim}{\mathrm{gdim}}

\newcommand{\la}{\lambda}

\newcommand{\Sym}{\mathfrak{S}}

\newcommand{\g}{\mathfrak{g}}

\newcommand{\gb}{\mathfrak{b}}

\newcommand{\h}{\mathfrak{h}}

\newcommand{\gp}{\mathfrak{p}}

\renewcommand{\P}{\mathbb{P}}

\newcommand{\Q}{\mathbb{Q}}

\newcommand{\bW}{\mathbb{W}}
\newcommand{\bX}{\mathbb{X}}

\newcommand{\Z}{\mathbb{Z}}

\title{Kostka polynomials of $G(\ell,1,m)$
}
\author{Syu \textsc{Kato}\footnote{Department of Mathematics, Kyoto University, Oiwake Kita-Shirakawa Sakyo Kyoto 606-8502 JAPAN \tt{E-mail:syuchan@math.kyoto-u.ac.jp}}}

\begin{document}
\maketitle

\begin{abstract}
For each integers $\ell > 1$ and $n \ge m \ge 1$, we prove an equivalence between the category of polynomial modules over a paraholic subalgebra $\mathfrak p$ of an affine Lie algebra of $\mathfrak{gl}(n\ell)$ and the module category of the smash product algebra $A$ of the complex reflection group $G(\ell,1,m)$ with $\mathbb C [X_1,\ldots,X_m]$. Then, we transfer the collection of $\mathfrak p$-modules considered in [Feigin-Makedonskyi-Khoroshkhin, arXiv:2311.12673] to $A$. By applying the Lusztig-Shoji algorithm [Shoji, Invent. Math. {\bf 74} (1983)], or rather its homological variant [K. Ann. Sci. ENS {\bf 48}(5) (2015)], we conclude that the multiplicity counts of these modules yield the Kostka polynomials attached to the limit symbols in the sense of [Shoji, ASPM {\bf 40} (2004)]. This result particularly settles a conjecture of Shoji [{\it loc. cit.} \S 3.13] and answers a question in [Shoji, Sci. China Math. {\bf 61} (2018)].
\end{abstract}
\section*{Introduction}
Kostka polynomials constitute a fundamental family of polynomials in representation theory. They appear in the description of characters of finite groups of Lie types (\cite{Gre55,Lus90b}), spherical functions of $p$-adic groups (\cite{Mac95,Lus81a}), the cohomologies of line bundles on the cotangent bundle of a flag manifold (\cite{Gup87}), and in the representation theory of the affinization of $\mathfrak{sl}$ (\cite{NY97}), to name a few.

Kostka polynomials associated with symmetric groups essentially refer to a unique family. However, this uniqueness does not hold if we replace the symmetric groups with other real or complex reflection groups. One can find such features in the context of generalized Springer correspondence (\cite{Lus84}), that is (further generalized and) systematically pursued by Shoji \cite{Sho02}. Shoji's generalized Kostka polynomials associated with the complex reflection group $G(\ell,1,m)$ contains a distinguished case, called the case of {\it limit symbols} \cite{Sho04}; some authors call them the Kostka-Shoji polynomials \cite{FI18,OS22}. This case conjecturally offers a honest generalization of the coinvariant realization of Kostka polynomials known for the case $\ell = 1$ (\cite{HS77,DP81,Tan82}), that is evidently unique. This conjecture has been proved for $\ell=2$ in \cite{Kat17}, but the rest of the cases are wide open.

The goal of this paper is to prove the above conjecture of Shoji completely by introducing a new realization of Kostka polynomials attached to the limit symbols. To state our results more precisely, we fix $m,\ell \in \Z_{>0}$ and the set $\Par_{m,\ell}$ of $\ell$-partitions of $m$.

\begin{fthm}[$\doteq$ Theorem \ref{thm:DU}]\label{fthm:maclim}
There exists a graded Lie algebra $\gp$ and its graded module category $\gp\mathchar`-\mathsf{gmod}_m$ with the following properties:
\begin{enumerate}
\item The isomorphism classes of irreducible modules in $\gp\mathchar`-\mathsf{gmod}_m$ are indexed by $\Par_{m,\ell}$ up to grading shifts;
\item For each $\vec{\la} \in \Par_{m,\ell}$, there exist modules $W_{\vec{\la}}$ and $W^{\flat}_{\vec{\la}}$ in  $\gp\mathchar`-\mathsf{gmod}_m$ that satisfy the orthogonality relations
$$\mathrm{Ext}^\bullet_{\gp\mathchar`-\mathsf{mod}_m} ( W^\flat_{\vec{\la}}, W_{\vec{\mu}^{\vee}}^{\star} ) = 0 \hskip 5mm \vec{\la}\neq \vec{\mu},$$
where $\star$ is an involution of $\gp\mathchar`-\mathsf{gmod}_m$ and $\vee$ is an involution on $\Par_{m,\ell}$;
\item Each $W_{\vec{\la}}$ has a simple head and a simple socle as graded $\gp$-modules.
\end{enumerate}
\end{fthm}

Mathematical components of Theorem \ref{fthm:maclim} are essentially contained in \cite{FKM,Fl21,FKM19,FMO23,FKMO,FKM23}. In this paper, we construct an appropriate category and extract its properties from these sources. We define
$$A := \C G(\ell,1,m) \ltimes \C [X_1,\ldots,X_m].$$
We equip $A$ with the grading $\deg X_i = 1$ and consider the category $A\mathchar`-\mathsf{gmod}$ of finitely generated graded $A$-modules.

\begin{fthm}[$\doteq$ Theorem \ref{thm:cateq}]\label{fthm:cateq}
There exists an equivalence of categories
$$\mathsf{SW} : \gp\mathchar`-\mathsf{gmod}_m \stackrel{\cong}{\longrightarrow} A\mathchar`-\mathsf{gmod}.$$
\end{fthm}

We transport the modules in Theorem \ref{fthm:maclim} 2) into $A\mathchar`-\mathsf{gmod}$ via the equivalence in Theorem \ref{fthm:cateq}. By applying the homological interpretation of the Lustzig-Shoji algorithm (\cite{Kat15}), we have:

\begin{fthm}[$\doteq$ Theorem \ref{thm:id} + Corollary \ref{cor:Kpos}]\label{fcor:koslim}
The graded multiplicities of simple modules in
$$\{\mathsf{SW} ( W_{\vec{\la}} )\}_{\vec{\la} \in \Par_{m,\ell}} \hskip 5mm \text{and} \hskip 5mm \{\mathsf{SW} ( W^{\flat}_{\vec{\la}} )\}_{\vec{\la} \in \Par_{m,\ell}}$$
are precisely the Kostka polynomials $K^-$ and $K^+$ attached to the limit symbols $(\cite{Sho04})$, respectively. In particular, for all $\vec{\mu},\vec{\la} \in \mathcal P_{m,\ell}$, we have
$$K^-_{{\vec{\mu}},\vec{\la}} (q), K^+_{{\vec{\mu}},\vec{\la}}(q) \in \Z_{\ge 0}[q].$$
\end{fthm}

For $\ell > 2$, the only previously known proof of the positivity of $K^-$ (\cite{FI18,Hu18,Sho18}) is based on a geometric realization \cite{FI18}, and the positivity of $K^+$ is new. Our construction provides an algebraic proof of the positivity of both Kostka polynomials attached to the limit symbols.

\begin{fcor}[$\doteq$ Corollary \ref{cor:Shoji} + Corollary \ref{cor:minchar}]\label{fcor:Shoji}
For each $\vec{\la} \in \Par_{m,\ell}$, the graded $A$-module $\mathsf{SW} ( W_{\vec{\la}} )^{\star}$ is the smallest quotient of $\C[X_1,\ldots,X_m]$ as graded $A$-modules that contains the socle of $\mathsf{SW} ( W_{\vec{\la}} )^{\star}$.
\end{fcor}

Corollary \ref{fcor:Shoji} settles a conjecture of Shoji \cite[\S 3.13]{Sho04}. In addition, it characterizes a family of polynomials and their duals, in a manner reminiscent to the case when $\ell = 1$ (\cite{HS77,DP81,Tan82}). Thus, we might say that
\begin{center}
the two families $K^{\pm}$ are {\it the} Kostka polynomials of $G(\ell,1,m)$,
\end{center}
setting aside numerous other cases (see e.g. \cite{Kat15,Hai03}).

Regarding this, we note a non-trivial aspect in our proof of Corollary \ref{fcor:koslim}: the partial order coming from the definition of $K^{\pm}$ {\it differs from} the partial order that naturally defines $\{ W_{\vec{\la}} \}_{\vec{\la}}$. Nevertheless, analysis of $W_{\vec{\la}}$ (implicitly) introduces a finer order that intertwines these two different partial orders.

Our construction leads to an answer to the question posed in \cite[Remark 2.16]{Sho18} (Remark \ref{rem:multigr}). In addition, we have an isomorphism $W_{\vec{\la}} \cong W_{\vec{\la}}^{\flat}$ ($\vec{\la} \in \Par_{m,2}$) offered by \cite{Kat17} that is non-trivial from our construction (Remark \ref{rem:inc}).

The organization of this paper is as follows: In section one, we recall the necessary results in this paper and rearrange them into forms suitable for our purpose. In particular, we explain Theorem \ref{fthm:maclim}. In section two, we prove Theorem \ref{fthm:cateq}. In section three, we analyze the structure of the modules $W_{\vec{\la}}$ and $W^{\flat}_{\vec{\la}}$ to prove Corollaries \ref{fcor:koslim} and \ref{fcor:Shoji}. We close this paper by exhibiting example calculations in section four.

\section{Preliminaries}
We work over $\C$. Throughout this paper, we fix positive integers $\ell > 1$ and $n \ge m \ge 1$. Let $\{\varepsilon_i\}_{i=1}^{n\ell}$ denote the standard basis of $\Z^{n\ell}$.
We define
$$\Z^{n\ell}_+ := \left( \sum_{1 \le i < j \le n \ell} \Z_{\ge 0} (\varepsilon_i - \varepsilon_j) \right) \setminus \{0\} \hskip 3mm \text{and} \hskip 3mm \Z^{n\ell}_{\ge 0} =: ( \Z_{\ge 0})^{n\ell}.$$
A graded vector space refers to a $\Z$-graded vector space unless stated otherwise. For a graded vector space $V = \bigoplus_{i \in \Z} V_i$, we define its graded dimension as
$$\gdim \, V := \sum_{i \in \Z} q^i \dim \, V_i.$$
For a graded vector space $M$, its restricted dual is defined to be
\begin{equation}
M^{\vee} := \bigoplus_{i \in \Z} (M_{-i})^* \hskip 5mm \text{where} \hskip 5mm M = \bigoplus_{i\in \Z} M_i.\label{eqn:restdual}
\end{equation}

For $a \in \Q$, we set
$$\lceil a \rceil := \min \{r \in \Z \mid r \ge a\} \hskip 5mm \text{and} \hskip 5mm \lfloor a \rfloor := \max \{r \in \Z \mid r \le a\}.$$
Let $\C [X] := \C[X_1,\ldots,X_{m}]$ and $\C [x^{\pm 1}] := \C[x_1^{\pm 1},\ldots,x_{n\ell}^{\pm 1}]$ be the polynomial ring in $m$-variables and the Laurent polynomial ring in $n\ell$-variables, respectively. For a (graded) module $M$ of a (graded) algebra $B$, the head $\mathsf{hd} \, M$ of $M$ is the largest semisimple (graded) $B$-quotient of $M$, and the socle $\mathsf{soc}\, M$ of $M$ is the largest semisimple (graded) $B$-submodule of $M$.

\subsection{Lie algebras}
Let $\Gamma := ( \Z / \ell \Z )$ and consider a $\Gamma$-module $V$ defined as:
$$V = \C^n \otimes \C \Gamma \cong \C^n \oplus ( \C^n \otimes \chi ) \oplus ( \C^n \otimes \chi^2 ) \oplus \cdots \oplus (\C^n \otimes \chi^{\ell -1}),$$
where $\chi : \Gamma \to \C^{\times}$ is the character that sends $a \in \Z$ to $e^{\frac{2a\pi\sqrt{-1}}{\ell}}$. We have
$$\mathfrak{gl}(n,\C)^{\oplus \ell} \equiv \mathfrak{gl}(n)^{\oplus \ell} = \mathrm{End} (\C^n)^{\oplus \ell} = \mathrm{End}_\Gamma (V) \subset \mathrm{End} (V) = \mathfrak{gl}(n \ell,\C) \equiv \mathfrak{gl}(n \ell).$$
The $\Gamma$-action on $V$ induces a $\Gamma$-action on $V \boxtimes V^* \cong \mathfrak{gl}(n \ell)$ that we denote by $\rho_{\Gamma}$. We set
$$\g^{\sharp} := \mathfrak{gl}(n\ell) \otimes \C [z]$$
and regard it as a graded Lie algebra by setting
$$\deg \, ( \xi \otimes z^m ) = m \hskip 5mm \text{for} \hskip 5mm \xi \in \mathfrak{gl}(n\ell), m \in \Z.$$
We consider the $\Gamma$-action on $\g^{\sharp}$ defined by
$$a.(\xi \otimes z^m ) := ( \rho_{\Gamma} (a)(\xi) \otimes e^{\frac{2am\pi\sqrt{-1}}{\ell}} z^m ) \hskip 5mm \xi \in \mathfrak{gl}(n\ell), a,m \in \Z.$$
Let $E_{ij} \in \mathfrak{gl}(n\ell)$ ($1\le i,j\le n\ell$) denotes the standard matrix unit. Then, $\{E_{ij} \otimes z^k\}_{1\le i,j \le n \ell, k \ge 0}$ forms a basis of $\g^\sharp$. Let $\gp$ denote the $\Gamma$-fixed part of $\g^{\sharp}$. We set $\gb_0 := \bigoplus_{i \le j} \C E_{ij},$ $\gb^{\sharp} := (\gb_0 + \g^{\sharp}z),$ and $\gb := (\gb^{\sharp} \cap \gp)$.

We define
$$\gp_0^\sharp := \bigoplus_{\lceil \frac{i}{n} \rceil - \lceil \frac{j}{n} \rceil \le 0} \C E_{ij} \subset \mathfrak{gl} (n\ell), \hskip 4mm \text{and} \hskip 4mm \gp^\sharp := \gp_0^\sharp + \gb^\sharp.$$
We have $(\gp_0^\sharp \cap \gp ) \cong \mathfrak{gl}(n)^{\oplus \ell}$ by inspection. 

\begin{lem}\label{lem:piso}
We have an isomorphism $\gp \cong \gp^{\sharp}$ as Lie algebras.	
\end{lem}

\begin{proof}
We have $E_{ij} \otimes z^{k} \in \gp$ if and only if
$$\lceil \frac{i}{n} \rceil - \lceil \frac{j}{n} \rceil + k \equiv 0 \mod \ell.$$
Thus, the map
\begin{equation}
\gp^{\sharp} \ni E_{ij} \otimes z^{k} \mapsto E_{ij} \otimes z^{\ell k - \lceil \frac{i}{n} \rceil + \lceil \frac{j}{n} \rceil} \in \gp\label{eqn:twist}
\end{equation}
is a well-defined injective map since
$$- \ell < \lceil \frac{i}{n} \rceil - \lceil \frac{j}{n} \rceil \le 0 \le \lceil \frac{j}{n} \rceil - \lceil \frac{i}{n} \rceil \hskip 5mm 1 \le i < j \le n \ell.$$
It is surjective and preserves the Lie bracket by inspection.
\end{proof}

Let $\Phi$ be the isomorphism in (\ref{eqn:twist}). By allowing negative powers of $z$, we have an ambient Lie algebra $\g \supset \gp$ obtained by sending $\g^{\sharp}$ by the same rule as in (\ref{eqn:twist}). Note that $\g$ is a Lie subalgebra of $\mathfrak{gl}(n\ell) \otimes \C [z,z^{-1}]$ instead of $\g^{\sharp}$.

Let $\Sym_{n\ell}$ be the symmetric group acting on $\C [x^{\pm 1}] $ by the permutation on indices of the variables $x_{\bullet}$. Let $\Sym$ denote the $\ell$-fold product of $\Sym_n$ that permutes the first $n$-indices, the second $n$-indices, $\ldots$ in $\{1,2,\ldots,n\ell\}$. We have $\Sym \subset \Sym_{n\ell}$. Let $u_0 \in \Sym$ be the longest element, that is the product of $\ell$-copies of longest elements of $\Sym_n$.

\begin{rem}\label{rem:grading}
The Lie algebra $\g$ is a multi-graded Lie algebra by introducing another degree $\mathtt{deg}$ valued in $\Z^{\ell}$ with its basis $\{\mathtt{e}_i\}_{i=1}^{\ell}$ as:
\begin{align*}
\mathtt{deg} \, E_{ij} \otimes z^{k\ell} & := k \mathtt{e}_1 + \cdots + k\mathtt{e}_{\ell} \hskip 5mm \text{if}\hskip 5mm \lceil \frac{i}{n} \rceil \equiv \lceil \frac{j}{n} \rceil, \hskip 2mm k \in \Z_{\ge 0}\\
\mathtt{deg} \, E_{ij} \otimes z^{- r_2 + r_1} & := \mathtt{e}_{r_1-1}+ \mathtt{e}_{r_1-2} + \cdots + \mathtt{e}_{r_2} \hskip 5mm \text{if}\hskip 5mm r_1 = \lceil \frac{j}{n} \rceil > r_2 = \lceil \frac{i}{n} \rceil\\
\mathtt{deg} \, E_{ij} \otimes z^{\ell - r_2 + r_1} & := \mathtt{e}_{r_1-1} + \cdots + \mathtt{e}_1 + \mathtt{e}_n + \cdots + \mathtt{e}_{r_2} \hskip 5mm \text{if}\hskip 5mm r_1 = \lceil \frac{j}{n} \rceil < r_2 = \lceil \frac{i}{n} \rceil.
\end{align*}
Note that $\mathtt{e}_i \mapsto 1$ ($1 \le i \le \ell$) recovers our grading.
\end{rem}

\subsection{Partitions}
We follow conventions in Macdonald \cite[Chap.~I]{Mac95} with the understanding that the length of partitions are $\le n$. For a partition $\la = ( \la_1,\la_2,\ldots,\la_n)$, we set $|\la| := \sum_{j=1}^n \la_n$ and call it the size of $\la$. Let $\Par_m$ denote the set of partitions of size $m$ (we have $m \le n$ by convention). The conjugate partition $\la'$ of $\la \in \Par_m$ is defined as $(\la')_i :=\{ j \ge 1 \mid \la_j \ge i \}$ for $1 \le i \le n$. We set
$$\mathsf{u} ( \la ) := \sum_{j = 1}^n \frac{\la_j(\la_j - 1)}{2} = \sum_{j = 1}^n (j-1)(\la')_j \hskip 10mm \la \in \Par_m,$$
where $\mathsf{n} ( \la ) = \mathsf{u} (\la')$ in the convention of \cite{Mac95}. For a partition $\la$ of $r$, let us denote $L_{\la}$ the irreducible $\Sym_{r}$-module under the convention
$$L_{(1^r)} = \mathsf{triv}, \hskip 5mm \text{and} \hskip 5mm L_{(r)} = \mathsf{sgn}.$$
This convention is not standard and the adjustment to the convention $L_{(r)} = \mathsf{triv}$ is given by tensoring $\mathsf{sgn}$, that results on taking the conjugate of a partition.

A collection of $\ell$-tuple of partitions $\vec{\la} = ( \la^{(1)},\la^{(2)},\ldots,\la^{(\ell)})$ is called a $\ell$-partition. The size $|\vec{\la}|$ of $\vec{\la}$ is understood to be $\sum_{i = 1}^{\ell} |\la^{(i)}|$. Let $\Par_{m,\ell}$ denote the set of $\ell$-partitions of size $m$. For a $\vec{\la} \in \Par_{m,\ell}$, we define the (ordered) sequence $\vec{\la}^{\mathtt{tot}} = \{\la_i^{\mathtt{tot}}\}_{i=1}^{n\ell} \in \Z_{\ge 0}^{n\ell}$ by listing the parts of each $\lambda^{(k)}$ in reverse order (from smallest to largest index), concatenated together:
$$\la^{(1)}_n,\la^{(1)}_{n-1},\ldots,\la^{(1)}_1,\la^{(2)}_n,\ldots,\la^{(2)}_1,\ldots,\la^{(\ell)}_n,\ldots,\la^{(\ell)}_1.$$
We define $\|\vec{\la}\| \in \Par_m$ as the permutation of $\{\la_i^{\mathtt{tot}}\}_i$ with $(\ell-1)n$ tuples of $0$ are omitted. For $\vec{\la} = ( \la^{(1)},\la^{(2)},\ldots,\la^{(\ell)}) \in \Par_{m,\ell}$, we define its componentwise conjugate as $\vec{\la}' := ( (\la^{(1)})',(\la^{(2)})',\ldots,(\la^{(\ell)})') \in \Par_{m,\ell}$.

The dominance order on $\Par_m$ is defined as follows:
$$\la \le \mu \hskip 5mm \text{if and only if}\hskip 5mm \sum_{j=1}^k \la_k \le \sum_{j=1}^k \mu_k \hskip 3mm \text{for each} \hskip 3mm 1 \le k \le n.$$

For $\vec{\la}, \vec{\mu} \in \Par_{m,\ell}$, we define a partial order $\lhd$ as:
$$\vec{\la} \lhd \vec{\mu} \hskip 5mm \text{if and only if}\hskip 5mm \|\vec{\la}\| < \|\vec{\mu}\| \hskip 5mm \text{or} \hskip 5mm \|\vec{\la}\| = \|\vec{\mu}\| \hskip 3mm \text{and} \hskip 3mm \vec{\la}^{\mathtt{tot}} - \vec{\mu}^{\mathtt{tot}} \in \Z^{n\ell}_+.$$

We define the $a$-function on $\Par_{m,\ell}$ as:
\begin{equation}
\mathsf{a} ( \vec{\la}) := \ell \sum_{j=1}^{\ell} \mathsf{u} ( \la^{(j)} ) + \sum_{j=1}^{\ell} (j-1) |\la^{(j)}|.\label{eqn:alim}
\end{equation}
We remark that the $\mathsf{a}$-value of $\vec{\la} \in \Par_{m,\ell}$ in \cite[(1.3)]{Sho04} is $\mathsf{a} ( \vec{\la}')$ in our notation.
\begin{ex}
The smallest element in $\Par_{m,\ell}$ is $((1^m)(\emptyset)^{\ell-1})$, and the largest element in $\Par_{m,\ell}$ is $((\emptyset)^{\ell-1}(m))$ with respect to $\lhd$. We have
$$\mathsf{a} ((1^m)(\emptyset)^{\ell-1}) =0, \hskip 5mm \text{and} \hskip 5mm \mathsf{a} ((\emptyset)^{\ell-1}(m)) = \ell \frac{m(m+1)}{2} - m.$$
\end{ex}

For $\vec{\la} \in \Par_{m,\ell}$, we set
\begin{align*}
(\vec{\la})^{\vee} & := ( (\la^{(1)}),(\la^{(\ell)}),(\la^{(\ell-1)}),\ldots,(\la^{(2)})) \in \Par_{m,\ell}, \\
(\vec{\la})^{*} & := ( (\la^{(\ell)}),(\la^{(\ell-1)}),\ldots,(\la^{(2)}),(\la^{(1)})) \in \Par_{m,\ell}, \hskip 5mm \text{and}\\
(\vec{\la})^{\theta} & := ( (\la^{(\ell)}),(\la^{(1)}),(\la^{(2)}),\ldots,(\la^{(\ell-1)})) \in \Par_{m,\ell}.
\end{align*}
\begin{lem}\label{lem:Pconj}
For each $\vec{\la} \in \Par_{m,\ell}$, we have
$$(\vec{\la})^{\vee} = ( \vec{\la}^* )^{\theta}, \hskip 5mm ( \vec{\la}^{\vee} )^{\vee} = \vec{\la}, \hskip 5mm \text{and} \hskip 5mm ( \vec{\la}^{*} )^{*} = \vec{\la}.$$
\end{lem}

\begin{proof}
Straight-forward.
\end{proof}

\begin{lem}\label{lem:order}
Let $\vec{\la}, \vec{\mu} \in \Par_{m,\ell}$ be such that $\|\vec{\la} \| =\| \vec{\mu}\|$. We fix $\la \in \Sym . \vec{\la}^{\mathtt{tot}}$, where $\Sym \subset \Sym_{n\ell}$ acts on $\Z^{n\ell}$ by the permutation of indices. If we have $\vec{\mu}^{\mathsf{tot}} \in ( \la + \Z^{n\ell}_{+})$, then we have $\vec{\mu} \unlhd \vec{\la}$.	
\end{lem}

\begin{proof}
For each $\sigma \in \Sym$, we have $\sigma \vec{\lambda}^{\mathsf{tot}} - \vec{\lambda}^{\mathsf{tot}} \in ( \Z^{n\ell}_{+} \cup \{0\} )$ by inspection. Thus, we have
$$\vec{\mu}^{\mathsf{tot}} - \vec{\la}^{\mathsf{tot}} = ( \vec{\mu}^{\mathsf{tot}} - \la) + ( \la - \vec{\mu}^{\mathsf{tot}}) \in \Z^{n\ell}_{+} + ( \Z^{n\ell}_{+} \cup \{0\}) \subset \Z^{n\ell}_{+}$$
as required.
\end{proof}

\subsection{Representations of Lie algebras}
We know that the action of $\g^\sharp z \subset \g^{\sharp}$ raises degree of an element of a $\Z$-graded modules of $\g^{\sharp}$. Since $\g^{\sharp}$ is concentrated in non-negative $\Z$-degree and a finitely generated graded $\g^{\sharp}$-module has bounded $\Z$-grading from the below, we find that an irreducible graded $\g^{\sharp}$-module is an irreducible $\mathfrak{gl}(n\ell)$-module regarded as a graded $\g^{\sharp}$-module through the projection
$$\g^{\sharp} \stackrel{z=0}{\longrightarrow} \mathfrak{gl}(n\ell).$$
The same is true for $\gp^\sharp$ by replacing $\mathfrak{gl}(n\ell)$ with $\mathfrak{gl}(n)^{\oplus \ell}$. Two Lie algebras $\mathfrak{gl}(n\ell)$ and $\mathfrak{gl}(n)^{\oplus \ell}$ share a maximal torus $\h := \bigoplus_{i=1}^{n\ell} \C E_{ii}$. We define
\begin{align*}
\bX^-_\ell :=\{ \{\la_i\}_i \in \Z^{n\ell} \mid \la_i \le \la_j \hskip 3mm \text{if} \hskip 3mm i < j \hskip 2mm \text{and} \hskip 2mm \lceil \frac{i}{n} \rceil = \lceil \frac{j}{n} \rceil\}.
\end{align*}
The set $\bX^-_{\ell}$ is identified with the set of anti-dominant integral weights of $\mathfrak{gl}(n)^{\oplus \ell}$. For $\la \in \bX^-_\ell$, we set $V_\la$ to be the irreducible finite-dimensional $\mathfrak{gl}(n)^{\oplus \ell}$-module generated from a $\h$-eigenvector $v_{\la}$ with its $\h$-weight $\la$ (i.e. $E_{ii}$ acts by $\la_i$ for $1 \le i \le n \ell$) by the $( \gb_0 \cap \gp )$-action. We embed $\Par_{m,\ell}$ into $\bX_{\ell}^-$ through $\vec{\la} \mapsto \vec{\la}^{\mathtt{tot}}$. Note that $\vec{\la}^{\mathtt{tot}} \in \Z^{n\ell}$ is the lowest $\h$-weight of $V_{\vec{\la}}$ (as $\mathfrak{gl}(n)^{\oplus \ell}$-module) by convention.

By the above discussion, we might regard $V_{\la}$ ($\la \in \bX^-_{\ell}$) as a simple graded $\gp$-module concentrated in degree $0$. We refer a finitely generated graded $\gp$-module, a finitely generated graded $\gp^{\sharp}$-module, or a finitely generated graded $\g^{\sharp}$-module $M$ as a polynomial module if $\h$ acts diagonally with their weights in $\Z_{\ge 0}^{n\ell}$.

We define $\gp\mathchar`-\mathsf{mod}_m$ (resp. $\gp^{\sharp}\mathchar`-\mathsf{mod}_m$ and $\g^{\sharp}\mathchar`-\mathsf{mod}_m$) to be the category of graded polynomial $\gp$-modules (resp. graded polynomial $\gp^{\sharp}$-modules and graded polynomial $\g^{\sharp}$-modules) such that $\mathrm{Id} \in \mathfrak{gl}(n\ell)$ acts by $m$. These categories admit autoequivalences $\mathsf{q}^{\pm 1}$ defined by
\begin{equation}
( \mathsf{q}^{\pm 1} M )_i = M_{i\pm 1} \hskip 5mm M = \bigoplus_{i \in \Z} M_i \in \g\mathchar`-\mathsf{mod}_m.\label{eqn:gshift}
\end{equation}

\begin{lem}
The isomorphism classes of simple objects in $\gp\mathchar`-\mathsf{gmod}_m$ is in bijection with $\Par_{m,\ell}$ up to grading shifts.
\end{lem}

\begin{proof}
The application of the strictly positive degree part $\gp_{>0}$ of $\gp$ defines a proper submodule of a (non-zero) module in $\gp\mathchar`-\mathsf{gmod}_m$. Hence, $\gp_{>0}$ must act on a simple module in $\gp\mathchar`-\mathsf{gmod}_m$ by zero. Thus, the assertion follows from the (polynomial) representation theory of $\mathfrak{gl}(n)^{\oplus \ell}$ (see e.g. \cite[p.~162 (8.2)]{Mac95}).
\end{proof}

For a finitely generated graded $\gp$-module $M$, we define its graded character by
$$\gch \, M := \sum_{i \in \Z, \mu \in \Z^{n\ell}} q^i x^{\mu} \dim \mathrm{Hom}_{\h} ( \C_{\mu}, M_i ) \in \C (\!(q)\!)[x^{\pm 1}],$$
where $x^{\mu} := \prod_{j=1}^{n\ell}x_j^{\mu_j}$, and $\C_{\mu}$ is the one-dimensional $\h$-module on which $E_{jj}$ acts by $\mu_j$ ($1 \le j \le n\ell$). For $\vec{\mu} \in \Par_{m,\ell}$, we set $x^{\vec{\mu}} := x^{\vec{\mu}^{\mathtt{tot}}}$. We define the character $\ch \, M$ of $M$ by setting $\ch \, M := \left. ( \gch \, M ) \right|_{q=1}$ whenever it makes sense.
 
\begin{lem}\label{lem:wt-est}
Let $\vec{\la}, \vec{\mu} \in \Par_{m,\ell}$. If the monomial $x^{\vec{\mu}}$ appears in $\ch \, V_{\vec{\la}}$, then we have $\vec{\la} \unrhd \vec{\mu}$. \hfill $\Box$
\end{lem}

\begin{proof}
The condition implies $\la^{(i)} \ge \mu^{(i)}$ and $|\la^{(i)}| = |\mu^{(i)}|$ for each $1\le i \le \ell$ by \cite[p.~101 (6.5)]{Mac95}. This implies $\vec{\la}\unrhd \vec{\mu}$.
\end{proof}

For $M \in \gp\mathchar`-\mathsf{gmod}_m$, its restricted dual $M^{\vee}$ is naturally a graded $\gp$-module.

\begin{lem}
There exists a Lie algebra automorphism $\theta$ of $\gp$ defined by:
\begin{eqnarray*}
\theta ( E_{ij} \otimes z^k ) = E_{i'j'} \otimes z^k, \hskip 25mm \text{where}\hskip 10mm\\
i' = \begin{cases} i + n & \text{if } i \leq n(\ell - 1) \\ i - n(\ell - 1) & \text{if } i > n(\ell - 1) \end{cases}, \quad j' = \begin{cases} j + n & \text{if } j \leq n(\ell - 1) \\ j - n(\ell - 1) & \text{if } j > n(\ell - 1) \end{cases}.
\end{eqnarray*}
In addition, we have $\theta^{\ell} = \mathrm{id}$.
\end{lem}

\begin{proof}
We find that $\theta$ induces a Lie algebra automorphism of $\g^{\sharp}$ that preserves $\gp$ by inspection. Note that our $\theta$ corresponds to the rotation by index $n$ in the affine Dynkin diagram of type $\mathsf{A}^{(1)}_{n\ell - 1}$. Thus, we have $\theta^{\ell} = \mathrm{id}$.
\end{proof}

For $M \in \gp\mathchar`-\mathsf{gmod}_m$, the vector space $M$ admits  the structure of a graded $\gp$-module twisted by $\theta^{-1}$. We denote this module by $M^{\theta}$.

\begin{lem}\label{lem:teff}
For each $k \in \Z$ and $\vec{\la} \in \Par_{m,\ell}$, we have $( \mathsf{q}^k V_{\vec{\la}})^{\theta} \cong \mathsf{q}^{k} V_{\vec{\la}^{\theta}}$.
\end{lem}

\begin{proof}
It suffice to notice that $\theta$ induces an automorphism of $\mathfrak{gl}(n)^{\oplus \ell}$ such that $\theta ( E_{ii} ) = E_{i+n,i+n}$ for each $1 \le i \le n \ell$, where the indices are understood to be modulo $n\ell$.
\end{proof}

\begin{lem}\label{lem:sigma}
Let $\sigma \in \Sym_{n\ell}$ be the permutation defined by
$$\sigma ( i ) =  i + n ( \ell - 2 \lceil \frac{i}{n} \rceil + 1) \hskip 10mm  (1 \le i \le n\ell).$$
The following assignment defines a Lie algebra automorphism $\phi$ of $\gp$:
$$E_{ij} \otimes z^{k} \mapsto - E_{\sigma ( j ) \sigma ( i )} \otimes z^{k} \hskip 10mm 1 \le i,j\le n\ell, k \ge 0.$$
\end{lem}

\begin{proof}
The assignment $\phi$ induces a Lie algebra automorphism of $\g^{\sharp}$ by inspection. It preserves $\gp$ since we have
$$\lceil \frac{i}{n} \rceil - \lceil \frac{j}{n} \rceil \equiv \lceil \frac{\sigma(j)}{n} \rceil - \lceil \frac{\sigma (i)}{n} \rceil \mod \ell$$
for each $1 \le i,j \le n \ell$.
\end{proof}

\begin{cor}
We have $\theta^{-1} \circ \phi = \phi \circ \theta$.
\end{cor}

\begin{proof}
This is by inspection.
\end{proof}

For $M \in \gp\mathchar`-\mathsf{gmod}_m$, the vector space $M^{\vee}$ admits the structure of a graded $\gp$-module, but it is not necessarily finitely generated nor a polynomial module. We set $M^{\star}$ as the same vector space $M^{\vee}$ with the $\gp$-action twisted by $\theta^{-1} \circ \phi$.

\begin{cor}\label{cor:tpcomm}
Let $M \in \gp\mathchar`-\mathsf{gmod}_m$ be such that $\dim \, M < \infty$. We have $M^{\star} \in \gp\mathchar`-\mathsf{gmod}_m$. In addition, we have $( M^{\star} )^{\star} \cong M$.
\end{cor}

\begin{proof}
The set of $\h$-weights on $M^{\vee}$ with respect to the honest dual action belongs to $- \Z_{\ge 0}^{n\ell}$, and the sum of the eigenvalues of $E_{ii}$ ($1\le i \le n \ell$) is $-m$. This is flipped by $\phi$, and preserved by $\theta$. Thus, the first assertion holds. The second assertion is a consequence of $(\theta^{-1}\circ\phi)^2 = \mathrm{id}$.
\end{proof}

\begin{lem}\label{lem:gsimple}
For each $k \in \Z$ and $\vec{\la} \in \Par_{m,\ell}$, we have
$$(\mathsf{q}^k V_{\vec{\la}})^{\star} \cong \mathsf{q}^{-k} ( V_{\vec{\la}} )^{\star}\cong \mathsf{q}^{-k} V_{\vec{\la}^{\vee}}.$$
\end{lem}

\begin{proof}
The duality preserves the simplicity of modules. The automorphisms $\theta$ and $\phi$ also preserve simplicity. Hence, it suffices to identify the labels. The flipping of the grading is coming from the duality. The isomorphism $( V_{\vec{\la}} )^{\star} \cong V_{\vec{\la}^{\vee}}$ follows from the comparison of $\h$-weights using Lemma \ref{lem:Pconj}.
\end{proof}

\subsection{Algebra $A$ and its representations}\label{subsec:algA}
We set $W := \Sym_m \ltimes (\Z / \ell \Z)^m$. Thanks to \cite[Chap.~I (7.4)]{Mac95}, the set $\mathsf{Irr} \, W$ of isomorphism classes of irreducible $W$-modules is in bijection with $\Par_{m,\ell}$.

We denote by $L_{\vec{\la}}$ the irreducible $W$-module corresponding to $\vec{\la}\in \Par_{m,\ell}$. This module is given by inducing the representation
\begin{equation}
L_{\la^{(1)}} \boxtimes \bigl( L_{\la^{(2)}} \otimes \chi^{\boxtimes |\la^{(2)}|} \bigr)\boxtimes \bigl( L_{\la^{(3)}} \otimes (\chi^2)^{\boxtimes |\la^{(3)}|} \bigr) \boxtimes \cdots \boxtimes \bigl(L_{\la^{(\ell)}} \otimes ( \chi^{\ell-1} )^{\boxtimes |\la^{(\ell)}|}\bigr),\label{eqn:preind}
\end{equation}
from $\prod_{j=1}^{\ell} ( \Sym_{|\la^{(j)}|} \ltimes \Gamma^{|\la^{(j)}|})$ to $W$, as detailed in \cite[Chap. I (9.4)]{Mac95}. To align with the convention of \cite{Sho02}, we need to tensorize the sign representation of $\Sym_{n\ell}$ extended to $W$ via the quotient map. This results in taking the componentwise conjugation $\vec{\la} \mapsto \vec{\la}'$ ($\vec{\la} \in \Par_{m,\ell}$) in the labelling set of $\mathsf{Irr}\, W$.

We define the smash product algebra
$$A := \C W \otimes \C [X_1,\ldots,X_m],$$
where $\C W \subset A$ is the group algebra of $W$, $\C [X_1,\ldots,X_m] \subset A$ is a polynomial algebra of $m$-variables, and the action of $(\sigma, \{e_i\}_{i=1}^m) \in \Sym_m \ltimes (\Z / \ell \Z)^m$ on the generators is given as:
\begin{equation}
(\sigma, \{e_i\}_i). X_j = e^{\frac{2\pi e_j \sqrt{-1}}{\ell}} X_{\sigma ( j )} . (\sigma, \{e_i\}_i)  \hskip 5mm 1 \le j \le m.\label{eqn:G-action}
\end{equation}

We equip the algebra $A$ with a grading by setting
$$\deg ( w \otimes 1 ) = 0 \hskip 5mm (w \in W), \hskip 5mm \deg ( 1 \otimes X_i ) = 1 \hskip 5mm (1 \le i \le m).$$
We regard $L_{\vec{\la}}$ ($\vec{\la} \in\Par_{m,\ell}$) as an irreducible graded $A$-module concentrated in degree $0$. Let $A\mathchar`-\mathsf{gmod}$ denote the category of finitely generated graded $A$-modules. The category $A\mathchar`-\mathsf{gmod}$ also admits a grading shift functors $\mathsf{q}^{\pm 1}$ defined by the same manner as in (\ref{eqn:gshift}).

\begin{lem}[\cite{Kat15} Corollary 2.3]
Up to grading shifts, the simple objects in $A\mathchar`-\mathsf{gmod}$ are precisely the modules $L_{\vec{\la}}$ $(\vec{\la} \in \Par_{m,\ell} )$, where the variables $\{X_i\}_{i=1}^m$ act trivially. \hfill $\Box$
\end{lem}

For $M \in A \mathchar`-\mathsf{gmod}$, the vector space $M^{\vee}$ (the restricted dual) naturally acquires the structure of a graded $A$-module, but it is not necessarily finitely generated. When $\dim \, M < \infty$, we have $(M^{\vee})^{\vee} \cong M$.

\begin{lem}\label{lem:Asimple}
For each $k \in \Z$ and $\vec{\la} \in \Par_{m,\ell}$, we have
$$(\mathsf{q}^k L_{\vec{\la}})^{\vee} \cong \mathsf{q}^{-k} ( L_{\vec{\la}} )^{\vee}\cong \mathsf{q}^{-k} L_{\vec{\la}^{\vee}}.$$
\end{lem}

\begin{proof}
The duality preserves the simplicity of modules and flips the grading. Hence, the assertion follows from
$$L_{\vec{\la}}^{\vee} \cong L_{\vec{\la}^{\vee}},$$
that is read out from (\ref{eqn:preind}) using $(\chi^i)^{*} \cong \chi^{\ell-i} \cong (\chi^*)^{i}$ ($0 \le i < \ell$) and $L_{\mu}^{*} \cong L_{\mu}$ as $\Sym_m$-modules for each partition $\mu$ of $m$ ($\Sym_m$ is a real reflection group).
\end{proof}

\subsection{Miscellaneous on representations}
For $\mathcal C := \g\mathchar`-\mathsf{gmod}_m$, $A\mathchar`-\mathsf{gmod}$, or other categories equipped with autoequivalences $\mathsf{q}^{\pm 1}$ that shifts the degree by $\pm 1$ (and $\mathsf{q} \circ \mathsf{q}^{-1} \cong \mathrm{id} \cong \mathsf{q}^{-1} \circ \mathsf{q}$), we define
$$\mathrm{ext}^i_{\mathcal C} ( M, N ) := \bigoplus_{j \in \Z} \mathrm{Ext}^i_{\mathcal C} ( M, \mathsf{q}^{-j} N ) \hskip 5mm M, N \in \mathcal C.$$
Then, each $\mathrm{ext}^i_{\mathcal C} ( M, N )$ is naturally a graded vector space. We regard $\mathrm{ext}^{\bullet}_{\mathcal C} (M,N)$ as the graded version of $\mathrm{Ext}^{\bullet}_{\mathcal C} (M,N)$. We may write $\mathrm{hom}_{\mathcal C}(M,N)$ instead of $\mathrm{ext}^0_{\mathcal C}(M,N)$. Note that $\mathrm{ext}^i_{\mathcal C} ( M, N )$ ($i \in \Z$), viewed as an ungraded vector space, recovers $\mathrm{Ext}^i (M,N)$ calculated in the corresponding ungraded category.

For $M \in \g\mathchar`-\mathsf{gmod}_m$ and $N \in A\mathchar`-\mathsf{gmod}$, we set
$$[M:V_{\vec{\la}}]_q := \gdim\,\mathrm{hom}_{\mathfrak{gl}(n)^{\oplus \ell}} (V_{\vec{\la}}, M) \hskip 4mm \text{and} \hskip 4mm [N:L_{\vec{\la}}]_q := \gdim\,\mathrm{hom}_{W} (L_{\vec{\la}}, N).$$

\begin{lem}\label{lem:multex}
Let $M \in \g\mathchar`-\mathsf{gmod}_m$ and $N \in A\mathchar`-\mathsf{gmod}$. We have
$$[M:V_{\vec{\la}}]_q, [N:L_{\vec{\la}}]_q \in \C(\!(q)\!) \hskip 5mm \vec{\la} \in \Par_{m,\ell}.$$
For short exact sequences
$$0 \to M_1 \to M_2 \to M_3 \to 0 \hskip 5mm \text{and} \hskip 5mm 0 \to M_1 \to M_2 \to M_3 \to 0$$
in $\gp\mathchar`-\mathsf{gmod}_m$ and $A\mathchar`-\mathsf{gmod}$, we have
$$[M_2 :V_{\vec{\la}}]_q = [M_1 :V_{\vec{\la}}]_q + [M_3 :V_{\vec{\la}}]_q\hskip 3mm \text{and} \hskip 3mm [N_2 :L_{\vec{\la}}]_q = [N_1 :L_{\vec{\la}}]_q + [N_3 :L_{\vec{\la}}]_q$$
for each $\vec{\la} \in \Par_{m,\ell}$, respectively.
\end{lem}

\begin{proof}
The multiplicity estimates hold since each graded component of $U(\gp)$ or $A$ is finite-dimensional. The rest of the assertions follow from the complete reducibility of (finite-dimensional) modules of $\mathfrak{gl}(n)$ and $W$.
\end{proof}

\begin{thm}[Specht \cite{Spe32}, see Regev \cite{Reg86} Corollary 7.9, cf. \cite{KP79} \S 6]\label{thm:cSW}
There is an isomorphism
$$V^{\otimes m} \cong \bigoplus_{\vec{\la} \in \Par_{m,\ell}} V_{\vec{\la}} \boxtimes L_{\vec{\la}}$$
of $(\mathfrak{gl}(n)^{\oplus \ell}, W)$-bimodules. \hfill $\Box$
\end{thm}

\subsection{Highest weight theory}
We review some materials in \cite{FMO23,FKM23}. We remark that although \cite{FKM23} does not deal with the case $\g^{\sharp} = \mathfrak{gl} (n\ell,\C[z])$, the arguments in \cite{FKM19,FMO23} applies to this case as well.

If $M \in \gp^{\sharp}\mathchar`-\mathsf{gmod}_m$ has simple head (resp. simple socle) sitting at degree $0$, then $\Phi ( M )$ acquires the structure of a module in $\gp\mathchar`-\mathsf{gmod}_m$ with its simple head (resp. simple socle) sitting at degree $0$ since the action of $\gp^{\sharp}_0$ uniquely determines the grading inside $M_0$ and $\gp$ is a $\Z^{\ell}$-graded Lie algebra (Remark \ref{rem:grading}). If $M, N \in \gp^{\sharp}\mathchar`-\mathsf{gmod}_m$ have either simple head or simple socle, then we have
\begin{equation}
\mathrm{ext}^{i}_{\gp^{\sharp}\mathchar`-\mathsf{gmod}_m} ( M, N ) \cong \mathrm{ext}^{i}_{\gp \mathchar`-\mathsf{gmod}_m} ( \Phi( M ), \Phi ( N ) )\label{eqn:ungradedExt}
\end{equation}
as {\it ungraded} vector spaces for each $i \in \Z$.

Our ordering $\lhd$ on $\Par_{m,\ell} \subset \bX^-_{\ell}$ coincides with the order induced from the Cherednik order (\cite[(1.16)]{Che95}) below $(0^{n\ell-1}m) = ((\emptyset)^{\ell-1}(m))^{\mathtt{tot}} \in \Z^{\oplus n \ell}$. Let $\widetilde{\mathfrak{sl}}(n\ell)$ be the affine Lie algebra of $\mathfrak{sl}(n\ell)$. We have a natural map
$$\gb^{\sharp} \longrightarrow \widetilde{\mathfrak{sl}}(n\ell)$$
that respects the grading. The image of this map contains the upper-triangular part of $\widetilde{\mathfrak{sl}}(n\ell)$ in the sense of \cite{Kac}. We set $\widetilde{\mathfrak{gl}}(n\ell) := ( \h + \widetilde{\mathfrak{sl}}(n\ell) )$. The following is essentially a reformulation of \cite[Theorem 6 and Theorem 7]{San00} except for the items 3) and 4). For 3), see \cite[Corollary 10.1]{Kac}. For 4), it is a well-known property of Demazure modules that follows by applying the PBW theorem to the definition of the Demazure modules (\cite{Kum02}):

\begin{thm}[Sanderson \cite{San00}]\label{thm:San}
For each $\la \in \Z^{n\ell}$, there exists a finite-dimensional $\gb^{\sharp}$-module $D_{\la}$ with the following properties:
\begin{enumerate}
\item $D_{\la}$ is generated from a degree zero vector with its $\h$-weight $\la$;
\item $D_{\la}$ is isomorphic to a $\gb^{\sharp}$-submodule of a level one integrable highest weight module of $\widetilde{\mathfrak{gl}}(n\ell)$ generated by an extremal weight vector $($i.e. a Demazure module of that level one integrable highest module$)$;
\item Every Demazure module of a level one integrable highest weight module of $\widetilde{\mathfrak{gl}}(n\ell)$ is isomorphic to some $D_\la$ up to grading shifts;
\item $D_{\la}$ is $\g^{\sharp}$-stable $($resp. $\gp^{\sharp}$-stable$)$ when $\la$ is antidominant, i.e. $\la_1 \le \la_2 \le \cdots \le \la_{n\ell}$ $($resp. $\la \in \bX^-_\ell)$;
\item $\gch \, D_{\la} = E_{\la} (x,q,0) \in \C[\![q]\!][x^{\pm 1}]$, where $E_{\la} (x,q,0)$ is the nonsymmetric Macdonald polynomial $(\rm{\cite{Che95}})$ of type $\mathfrak{gl}(n\ell)$ specialized to $t=0$.
\end{enumerate}
\end{thm}

\begin{cor}\label{cor:simple}
For each $\vec{\la} \in \Par_{m,\ell}$, we set $D_{\vec{\la}} := D_{\vec{\la}^{\mathtt{tot}}}$. The module $D_{\vec{\la}}$ is $\gp^{\sharp}$-stable, and has simple head $V_{\vec{\la}}$ and simple socle $V_{((1^{m})(\emptyset)^{\ell-1})}$ as $\gp^{\sharp}$-modules.
\end{cor}

\begin{proof}
The level one dominant integral weight of $\widetilde{\mathfrak{sl}}(n\ell)$ in Theorem \ref{thm:San} 2) must have $\h$-weight $(1^m0^{n\ell-m})$ since this is the unique level one dominant affine weight in $\vec{\la}^{\mathtt{tot}} + \Z^{n\ell}_+$. We have $\vec{\la}^{\mathtt{tot}} \in \bX_\ell^-$. With these in hands, the assertion follows from Theorem \ref{thm:San} with the trivial action of $\mathrm{Id} \otimes\C [\![z]\!]z \subset \gp^\sharp$.
\end{proof}

\begin{prop}\label{prop:diag}
For each $\vec{\la} \in \Par_{m,\ell}$, we have
\begin{equation}
\phi^{-1} ( D_{-u_0 \vec{\mu}^{\mathtt{tot}}} ) \cong D_{\vec{\mu}^*}.\label{eqn:diagdem}
\end{equation}
\end{prop}

\begin{proof}
The algebra automorphism $\phi$ extends to the whole affine Lie algebra $\widetilde{\mathfrak{gl}} ( n\ell )$ that preserves the center and the grading operator. Hence, $\phi$ sends a level one integrable highest weight modules of $\widetilde{\mathfrak{gl}} ( n\ell )$ to a level one integrable highest weight modules of $\widetilde{\mathfrak{gl}} ( n\ell )$. Since $\phi (\gp^{\sharp} ) = \gp^{\sharp}$, it also sends a $\gp^{\sharp}$-stable Demazure submodule (of an integrable highest weight module of $\widetilde{\mathfrak{gl}} ( n\ell )$) to some $\gp^{\sharp}$-stable Demazure submodule (of some integrable highest weight module with the same level). Therefore, we conclude (\ref{eqn:diagdem}) by Theorem \ref{thm:San} 3).
\end{proof}

The following result is extracted from \cite{FKM,FKM19,FMO23,FKMO,FKM23}:

\begin{thm}\label{thm:DU}
For each $\vec{\la} \in \Par_{\ell,m}$, there exist three modules $\mathbb D_{\vec{\la}}, \mathbb U_{\vec{\la}}, U_{\vec{\la}} \in \gp^{\sharp}\mathchar`-\mathsf{gmod}_m$ with the following properties:
\begin{enumerate}
\item If we have $[U_{\vec{\la}} : V_{\vec{\mu}^*}]_q \neq \delta_{\vec{\la},\vec{\mu}}$, then we have $\vec{\mu}^* \lhd \vec{\la}^*$ for each $\vec{\la},\vec{\mu} \in \Par_{m,\ell};$
\item The projective cover $\P_{\vec{\la}}^{\sharp}$ of $V_{\vec{\la}}$ in $\gp^{\sharp}\mathchar`-\mathsf{gmod}_m$ admits a finite filtration by $\{\mathsf{q}^j \mathbb U_{\vec{\la}}\}_{j \in \Z,\vec{\la} \in \Par_{m,\ell}}$. It also admits a finite filtration by $\{\mathsf{q}^j \mathbb D_{\vec{\la}}\}_{j \in \Z,\vec{\la} \in \Par_{m,\ell}}$;
\item We have $\mathrm{ext}^{i}_{\gp^{\sharp}\mathchar`-\mathsf{gmod}_m} ( \mathbb U_{\vec{\la}}^{\theta}, D_{\vec{\mu}}^{\star} ) \cong \C^{\delta_{i0}\delta_{\vec{\la}, \vec{\mu}}}$ for each $\vec{\la},\vec{\mu} \in \Par_{m,\ell}$;
\item We have a graded polynomial ring $R_{\vec{\la}}$ of at most $m$ variables such that
$$\mathrm{End}_{\gp^{\sharp}} ( \mathbb{U}_{\vec{\la}} ) \cong \mathrm{End}_{\gp^{\sharp}} ( \mathbb{D}_{\vec{\la}^{*}} ) \cong R_{\vec{\la}};$$
\item The modules $\mathbb{U}_{\vec{\la}}$ and $\mathbb{D}_{\vec{\la}^{*}}$ are free of finite rank over $R_{\vec{\la}}$;
\item If $R$ is a finite-dimensional graded quotient ring of $R_{\vec{\la}}$, then $R \otimes_{R_{\vec{\la}}}\mathbb{D}_{\vec{\la}^{*}}$ $($resp. $R \otimes_{R_{\vec{\la}}}\mathbb{U}_{\vec{\la}})$ is a finite extension of graded shifts of $D_{\vec{\la}^{*}}$ $($resp. $U_{\vec{\la}})$.
\end{enumerate}
\end{thm}

\begin{proof}
We set
$$J := \{1,\ldots,n\ell\} \setminus \{n,2n,\ldots,n\ell\} \subset \{1,\ldots,n\ell-1\}$$
and regard it as a subset of the Dynkin index of root system of type $\mathsf{A}_{n\ell - 1}$. Note that $\gp^{\sharp}$ is the paraholic algebra corresponding to $J$ in \cite{FKM23} (up to the difference between $\mathfrak{sl}$ and $\mathfrak{gl}$ that is covered by the same recipe as in \cite{FKM19,FMO23}).

In \cite[Lemma 4.34]{FKM23}, we have graded $\gp^\sharp$-modules $D_{\vec{\la}^{\mathrm{tot}}}^{J}$, $\mathbb D_{\vec{\la}^{\mathrm{tot}}}^{J}$, and $U_{\vec{\la}^{\mathrm{tot}}}^J$, and $\mathbb U_{\vec{\la}^{\mathrm{tot}}}^J$ ($\vec{\la} \in \Par_{m,\ell}$) by $\vec{\la}^{\mathtt{tot}} \in \bX^-_{\ell}$. We have $D_{\vec{\la}} = D_{\vec{\la}^{\mathrm{tot}}}^{J}$ (\cite[Remark 3.1]{FKM}). We set
\begin{equation}
\mathbb D_{\vec{\la}} = \mathbb D_{\vec{\la}^{\mathrm{tot}}}^{J}, \hskip 5mm U_{\vec{\la}} := \phi ^{-1} ( U_{-u_0 \vec{\la}^{\mathrm{tot}}}^{J} ), \hskip 5mm \text{and} \hskip 5mm \mathbb U_{\vec{\la}} := \phi^{-1} ( \mathbb U_{-u_0\vec{\la}^{\mathrm{tot}}}^{J} ).\label{eqn:defmods}
\end{equation}
Then, the first item follows from \cite[Lemma 4.34]{FKM23} (cf. Proposition \ref{prop:diag}).

Note that $\Par_{m,\ell} \subset \bX_\ell^-$ is the set of all weights in $\bX^-_{\ell}$ that is lower than $((\emptyset)^{\ell-1}(m))$ by the Cherednik order (\cite[Definition 2.1]{FKM23}). In view of \cite[Theorem 4.35]{FKM23}, we find that the $\gp\mathchar`-\mathsf{gmod}_m$ is obtained as the truncation of $\gp\mathchar`-\mathsf{gmod}$ with respect to the Cherednik order (\cite[\S 3.1]{FKM23}). Thus, the second item hold as the $\phi$-twists of \cite[Proposition 1.19 and Corollary 1.18]{FKMO}. The third item follows from \cite[Proposition 1.19 and Theorem 1.11]{FKMO} and Proposition \ref{prop:diag} by the $\theta$-twist ($\star$ involves the $\theta$-twists).

We deduce the fourth item by \cite[Lemma 4.32]{FKM23} and \cite[Corollary 4.15]{FKMO}. By \cite[Corollary 5.3]{FKM19}, we find that $R_{\vec{\la}}$ is precisely the multiplicity of $V_{\vec{\la}}$ in $\mathbb{U}_{\vec{\la}}$ (resp. the multiplicity of $V_{\vec{\la}^{*}}$ in $\mathbb{D}_{\vec{\la}^{*}}$). Here $\mathbb{D}_{\vec{\la}^{*}}$ and $\mathbb{U}_{\vec{\la}}$ are obtained as the (limit of) successive self-extensions of $D_{\vec{\la}^{*}}$ and $U_{\vec{\la}}$ by \cite[Theorem 4.35 and Corollary 5.5]{FKM23}, respectively. From this, we conclude the fifth and sixth items (cf. \cite[Proof of Proposition 3.9]{Kat17}).
\end{proof}

\begin{cor}\label{cor:gch}
For each $\vec{\la} \in \Par_{m,\ell}$, we have
\begin{equation}
\gch \, U_{\vec{\la}^*} = \sum_{\mu \in \Sym . \vec{\la}^{\mathtt{tot}}} c_\mu (q) E_{- \mu} ( x_1^{-1},\ldots,x_{n\ell}^{-1},q^{-1}, \infty ) \hskip 5mm c_\mu (q) \in \mathbb Z [\![q]\!]^{\times}, \label{eqn:gchU}
\end{equation}
where $\Sym$ acts on $\Z^{n\ell}$ by the permutation of indices, $-$ acts on $\mu \in \Z^{n\ell}$ as $\mu_i \mapsto - \mu_{i}$ for $1 \le i \le n \ell$.
\end{cor}

\begin{proof}
Taking (\ref{eqn:defmods}) into account, the assertion follows from the comparison of \cite[Corollaries 5.5 and 5.7]{FKM23} and Lemma \ref{lem:order}.
\end{proof}

\subsection{Recollections on the induction functor}

For a (graded) $\h$-semisimple $\gb^{\sharp}$-module $M$, we define
$$\mathbb L^{-i} \mathscr D (M) := H^i ( G / B, \mathcal E ( M ) )^{\vee} \hskip 5mm i \ge 0,$$
where $G= \mathop{GL}(n\ell)$, $B \subset G$ is the Borel subgroup whose Lie algebra yields $\gb_0^{\sharp} \subset \mathfrak{gl}(n\ell)$, and $\mathcal E (M)$ is the vector bundle whose total space is defined as $G \times^B M$. We write $\mathscr D (M)$ instead of $\mathbb L^0 \mathscr D (M)$. We have a surjective map
\begin{equation}
U (\g) \otimes_{U(\gb)}M \longrightarrow \!\!\!\!\! \rightarrow \mathscr D (M)\label{eqn:algDem}
\end{equation}
coming from the algebraic interpretation in \cite[Chapter 4]{Jan03}.

\begin{thm}\label{thm:sDem}
Let $M$ be a graded $\h$-semisimple $\gb^{\sharp}$-module whose graded pieces are finite-dimensional. We have
\begin{enumerate}
\item The functor $\mathscr D$ is right exact, and $\mathbb L^{\bullet} \mathscr D$ is its left derived functor. In particular, a short exact sequence
$$0 \rightarrow M_1 \rightarrow M \rightarrow M_2 \rightarrow 0$$
of graded $\h$-semisimple $\gb^{\sharp}$-modules yields a long exact sequence
$$\cdots \rightarrow \mathbb L^{i-1} \mathscr D (M_2) \rightarrow \mathbb L^i \mathscr D (M_1) \rightarrow \mathbb L^i \mathscr D (M) \rightarrow \mathbb L^i \mathscr D (M_2) \rightarrow  \cdots;$$
\item Assume that $M$ acquires the structure of a $\g^{\sharp}$-module. Then, we have
$$\mathbb L^\bullet \mathscr D (M) \equiv \mathscr D (M) \cong M;$$
\item We have a natural transformation $\mathrm{Id} \rightarrow \mathscr D$ that induces $\mathscr D \stackrel{\cong}{\longrightarrow} \mathscr D \circ \mathscr D$;
\item The functor $\mathscr D$ is the composition of Demazure functors $(\cite{Jos85})$ with respect to a reduced expression of the longest element $w_0$ of $\Sym_{n\ell}$.
\end{enumerate}
In addition to this, $\mathscr D (M)$ is a graded $\g^{\sharp}$-module. 
\end{thm}

\begin{proof}
The first assertion follows from generality on the cohomology of algebraic varieties (see e.g. \cite[Chap.~I\!I\!I]{Har77}). Since the construction commutes the direct sum decomposition with respect to the grading, we find that $\mathscr D$ can be applied graded componentwise. Hence, \cite[8.1.19 Theorem (4)]{Kum02} applied to $G/B$ yields the fourth assertion. From this, the second and the third assertions follow from a repeated applications of \cite[\S 2.7]{Jos85}.

The affine flag variety of $\widetilde{\mathfrak{sl}(n\ell)}$ contains $G/B$ as its Schubert variety corresponding to $w_0$. Hence, we can identify the effect of Demazure functors of $G$ and that of affine Kac-Moody group of $\widetilde{\mathfrak{sl}(n\ell)}$ for a reduced expression of $w_0$. Thus, we find that $\mathscr D (M)$ is a module of
$$\g^{\sharp} = \left< \gb^{\sharp}, E_{i,i-1} \mid 1 < i \le n\ell \right>$$
by \cite[\S 2.7]{Jos85}, that is graded as $\mathscr D$ can be applied graded componentwise.
\end{proof}

\begin{prop}\label{prop:Dsym}
For each $\vec{\la} \in \Par_{m,\ell}$, we have
$$\mathbb D_{\vec{\la}} \subset \mathscr D ( \mathbb D_{\vec{\la}} ) \equiv \mathbb L^{\bullet} \mathscr D ( \mathbb D_{\vec{\la}} ).$$
\end{prop}

\begin{proof}
Each $D_{\vec{\la}}$ is a level one affine Demazure module by Theorem \ref{thm:San}. Thus, \cite[8.1.8 and 8.1.13 Theorem]{Kum02} implies that
$$\mathbb L^{\bullet} \mathscr D ( D_{\vec{\la}} ) \equiv \mathscr D ( D_{\vec{\la}} ).$$
\cite[8.1.11 Corollary]{Kum02} further yields $D_{\vec{\la}} \subset \mathscr D ( D_{\vec{\la}} )$. Theorem \ref{thm:DU} 6) and the fact that our functor $\mathscr D$ can be applied graded componentwise implies that we can replace $D_{\vec{\la}}$ with $\mathbb D_{\vec{\la}}$
using Theorem \ref{thm:sDem} 1).
\end{proof}

\begin{cor}\label{cor:DP}
For each $\vec{\la} \in \Par_{m,\ell}$, we have $\P_{\vec{\la}}^\sharp \subset \mathscr D ( \P_{\vec{\la}}^\sharp )$. If we have $\vec{\la} = ( (\emptyset)^{\ell-1},\la)$ for a partition $\la$, then we have $\P_{\vec{\la}}^\sharp \cong \mathscr D ( \P_{\vec{\la}}^\sharp )$.
\end{cor}

\begin{proof}
By Theorem \ref{thm:DU} 2) and Proposition \ref{prop:Dsym}, we find
$$\P_{\vec{\la}}^\sharp \subset \mathscr D ( \P_{\vec{\la}}^\sharp )$$
using Theorem \ref{thm:sDem} 1) repeatedly. This is the first assertion.

In case $\vec{\la} = ( (\emptyset)^{\ell-1},\la)$, the module $D_{\vec{\la}}$ is $\g^{\sharp}$-stable. In particular, the degree zero part $(\P_{\vec{\la}}^\sharp)_0$ of $\P_{\vec{\la}}^\sharp$ is stable under the action of $\mathfrak{gl} (n\ell)$. We have $\gp^{\sharp} = \g^{\sharp} z \oplus \gp^{\sharp}_{0}$ and $\g^{\sharp} = ( \mathfrak{gl} (n\ell) + \gp^\sharp )$. It follows that $\P_{\vec{\la}}^\sharp$ acquires the structure of a $\g^{\sharp}$-module since $\P_{\vec{\la}}^\sharp$ is a quotient of
$$U ( \gp^{\sharp} ) \otimes_{U (\gp_0^{\sharp})} ( \P_{\vec{\la}}^\sharp )_0 \cong U ( \g^{\sharp} z ) \otimes_{\C} ( \P_{\vec{\la}}^\sharp )_0 \cong U ( \g^{\sharp} ) \otimes_{U (\mathfrak{gl}(n\ell))} ( \P_{\vec{\la}}^\sharp )_0,$$
and the kernel is described by the weight conditions that is stable under the action of $\h$ and $\Sym_{n\ell}$ (note that $\mathfrak{gl}(n\ell)$ is spanned by the $\Sym_{n\ell}$-translates of $\mathfrak{gl}(n)^{\oplus \ell}$). This forces $\P_{\vec{\la}}^\sharp \cong \mathscr D ( \P_{\vec{\la}}^\sharp )$ by Theorem \ref{thm:sDem} 2). This is the second assertion.
\end{proof}

\subsection{The Lusztig-Shoji algorithm}
In \cite{Sho02}, the so-called Lusztig-Shoji algorithm (\cite{Sho83,Lus84}) was generalized to the case of complex reflection groups. We explain its module-theoretic interpretation (\cite{Kat15}). Recall that we have defined the $\mathsf{a}$-function on $\Par_{m,\ell}$, that is equivalent to defining the $\mathsf{a}$-function on $\mathsf{Irr}\, W$.

Let $P_{\vec{\la}}$ be the (graded) projective cover of $L_{\vec{\la}}$ in $A \mathchar`-\mathsf{gmod}$ (see e.g. \cite[\S 2]{Kat15}).

We define a $\Q(\!(q)\!)$-valued square matrix $\Omega$ of size $|\mathsf{Irr}\, W|$ as:
$$\Omega_{\vec{\la},\vec{\mu}} := \sum_{i \ge 0} q^i \dim \mathrm{Hom}_{W} ( L_{\vec{\mu}}, ( P_{\vec{\la}} )_i) \in \Q(\!(q)\!).$$

We define two $\Q(\!(q)\!)$-valued unknown square matrices $K^{\pm}$ of size $|\mathsf{Irr}\, W|$ as:
$$K^+_{\vec{\la},\vec{\mu}} = \delta_{\vec{\la}, \vec{\mu}} \hskip 3mm \text{if} \hskip 3mm \mathsf{a}(\vec{\la}^{\vee}) \ge \mathsf{a}(\vec{\mu}^{\vee}) \hskip 5mm \text{and} \hskip 5mm K^-_{\vec{\la},\vec{\mu}} = \delta_{\vec{\la}, \vec{\mu}} \hskip 3mm \text{if} \hskip 3mm \mathsf{a}(\vec{\la}) \ge \mathsf{a}(\vec{\mu}).$$
We also consider an unknown $\Q(\!(q)\!)$-valued block-diagonal matrix $\Lambda$ of size $|\mathsf{Irr}\, W|$ as:
$$\Lambda_{\vec{\la},\vec{\mu}} = 0 \hskip 5mm \text{if} \hskip 5mm \mathsf{a}( \vec{\la} ) \neq \mathsf{a}( \vec{\mu} ).$$
Let $K^{\sigma}$ be the permutation of the matrix $K$ by means of $(\vec{\la},\vec{\mu}) \mapsto ((\vec{\la})^{\vee},(\vec{\mu})^{\vee})$.

The following is a (reformulation of the) special case of \cite[Theorem 2.10]{Kat15} with respect to the partial order\footnote{When we apply \cite[Theorem 2.10]{Kat15} to the generalized Springer correspondence \cite{Lus84}, we have another partial order $\prec$ on $\mathsf{Irr}\, W$ such that $\vec{\la} \prec \vec{\mu}$ implies $\mathsf{a} ( \vec{\la} ) > \mathsf{a} ( \vec{\mu} ) $.} induced from the $\mathsf{a}$-function:

\begin{thm}[Shoji, Lusztig, see \cite{Kat15} Theorem 2.10]\label{thm:LSalg}
In the above settings, the matrix equation
\begin{equation}
( K^+ )^{\sigma} \cdot \Lambda \cdot {}^{\mathtt t} K^- = \Omega\label{eqn:LSalg}
\end{equation}
has a unique solution. \hfill $\Box$
\end{thm}

\begin{thm}[\cite{Kat15} Theorem 2.17]\label{thm:homLSalg}
Keep the setting of Theorem $\ref{thm:LSalg}$. If the matrices $K^{\pm}$ records the graded characters of modules $\{K^{\pm}_{\vec{\la}}\}_{\vec{\la} \in \Par_{m,\ell}}$ in $A\mathchar`-\mathsf{gmod}$ as
$$K^{\pm}_{\vec{\la},\vec{\mu}} (q) = \gdim\, \mathrm{hom}_W ( L_{\vec{\la}}, K^{\pm}_{\vec{\mu}} ) \hskip 10mm \vec{\la},\vec{\mu} \in \Par_{m,\ell},$$
and we have
$$\mathrm{ext}^\bullet_{A\mathchar`-\mathsf{gmod}} ( K^+_{\vec{\la}}, ( K^-_{(\vec{\mu})^{\vee}} )^{\vee} ) = 0 \hskip 10mm \vec{\la} \neq \vec{\mu},$$
then $K^{\pm}$ satisfies $(\ref{eqn:LSalg})$. \hfill $\Box$
\end{thm}

\section{A variant of the Schur-Weyl duality}
Keep the setting of the previous section. We have a natural $\g^{\sharp}$-module structure on
$$\bW_1 := V \otimes \C [z].$$

We have an action of $\gp$ on $\bW_1^{\otimes m}$ that commutes with the $\Sym_m$-action permuting tensor components by restriction. The $\Gamma$-action on $\bW_1$ commutes with the $\gp$-action. Hence, it gives rise to the $\Gamma^m$-action on $\bW_1^{\otimes m}$ that commutes with $\gp$. It follows that $\bW_1^{\otimes m}$ is a $(\gp, A)$-bimodule.

\begin{lem}\label{lem:Wproj}
The module $\bW_1^{\otimes m}$ is projective as $A$-modules.	
\end{lem}

\begin{proof}
The action of $\C [X]$ makes $\bW_1^{\otimes m}$ into a free module of rank $( n\ell ) \cdot m$ generated by its degree zero part. By \cite[Lemma 2.2]{Kat15}, such a module is necessarily projective as a (graded) $A$-module.
\end{proof}

Let $A^\flat$ be the graded subalgebra of $A$ generated by $\Sym_m\subset W$ and $\C[X]$.

\begin{thm}[Feigin-Flicker-Khoroshkhin-Makedonskyi \cite{FKM19,Fl21}]\label{thm:Usurj}
We have a surjection
$$U ( \g^{\sharp} ) \longrightarrow \!\!\!\!\! \rightarrow \mathrm{End}_{A^\flat} ( \bW_1^{\otimes m}),$$
and $\bW_1^{\otimes m}$ is a progenerator in $\g^{\sharp}  \mathchar`-\mathsf{gmod}_m$.
\end{thm}

\begin{prop}\label{prop:gen}
The module $\bW_1^{\otimes m}$ is generated by its degree zero part  as $U ( \gp )$-modules.
\end{prop}

\begin{proof}
We might write $v \otimes z^r$ ($v \in V$, $r \ge 0$) by $vz^r$ for the sake of simplicity during this proof. The standard basis
$$e_1,\ldots,e_{m\ell} \in V = \C^{n\ell}$$
induces a basis of $\bW_1^{\otimes m}$ as:
\begin{equation}
e_{i_1} z^{r_1} \otimes \cdots \otimes e_{i_m} z^{r_m} \hskip 5mm 1 \le i_1,\ldots,i_{m} \le n\ell, r_1,\ldots,r_m \ge 0.\label{eqn:ebasis}
\end{equation}
If $r_j > 0$ for $1 \le j \le m$, then we find $1 \le i_{j}' \le n\ell$ such that
$$\lceil \frac{i'_{j}}{n} \rceil = \lceil \frac{i_{j}}{n} \rceil + 1 \mod \ell \hskip 5mm \text{and} \hskip 5mm i'_{j} \not\in \{i_1,\ldots,i_{m}\} \setminus \{i_j\}$$
by the pigeon hole principle ($m \le n$). It follows that $( E_{i_{i_j},i_{i_j}'} \otimes z) \in \gp$ and
$$e_{i_1} z^{r_1} \otimes \cdots \otimes e_{i_m} z^{r_m} = ( E_{i_{j},i_{j}'}\otimes z )\left( e_{i_1} z^{r_1} \otimes \cdots \otimes e_{i'_j} z^{r_j-1} \otimes \cdots \otimes e_{i_m}  z^{r_m} \right).$$
Thus, the assertion holds.
\end{proof}

\begin{lem}\label{lem:multcount}
We have an isomorphism
$$\mathrm{End}_{\gp} ( \bW_1^{\otimes m} )\cong \mathrm{Hom}_{\mathfrak{gl}(n)^{\oplus \ell}} ( V^{\otimes m}, \bW_1^{\otimes m}) = \bigoplus_{\vec{\la} \in \Par_{m,\ell}}\mathrm{Hom}_{\mathfrak{gl}(n)^{\oplus \ell}} ( V_{\vec{\la}} \boxtimes L_{\vec{\la}}, \bW_1^{\otimes m}).$$
\end{lem}

\begin{proof}
The second equality is by Theorem \ref{thm:cSW}. By Proposition \ref{prop:gen}, we have an inclusion of the LHS to the RHS. Since the multiplication by $\C[z]$ defines a $U(\g^{\sharp})$-endomorphism of $\bW_1$, the multiplication of $\C[z]$ in the $i$-th tensor factor ($1 \le i \le m$) of $\bW_1^{\otimes m}$ also defines a $U(\g^{\sharp})$-endomorphism. This exhausts the RHS from the degree zero part
$$\mathrm{End}_{\mathfrak{gl}(n)^{\oplus \ell}} ( V^{\otimes m}),$$
in which we find an identification.
\end{proof}

Using Lemma \ref{lem:piso}, we define $\P_{\vec{\la}} := \Phi ( \P_{\vec{\la}}^{\sharp} )$ for each $\vec{\la} \in \Par_{m,\ell}$. The module $\P_{\vec{\la}}$ is the projective cover of $V_{\vec{\la}}$ in $\gp \mathchar`-\mathsf{gmod}_m$ by its definition (in Theorem \ref{thm:DU}).

\begin{lem}\label{lem:Pd}
For each $\vec{\la} = ( \la^{(1)},\ldots,\la^{(\ell)}) \in \Par_{m,\ell}$ such that $|\la^{(\ell)}| = m$, the projective module $\P_{\vec{\la}}$ defines a submodule of $\bW_1^{\otimes m}$ that has non-zero image in $V^{\otimes m} = \mathsf{hd} \, \bW_1^{\otimes m}$ as $\gp$-module.
\end{lem}

\begin{proof}
We set $\vec{\la} = ((\emptyset)^{\ell-1},\la)$. By Corollary \ref{cor:DP}, we find that $\P_{\vec{\la}}$ is in fact a $\g$-module. Thus, $\P_{\vec{\la}}$ is also projective as $\g$-module since a $\g$-projective cover of an irreducible $\mathfrak{gl}(n\ell)$-module with its lowest weight $\vec{\la}^{\mathtt{tot}}$ is generated by $v_{\vec{\la}^{\mathtt{tot}}}$ as a $\gp$-module (by the PBW theorem). Now Theorem \ref{thm:Usurj} implies that $\P_{\vec{\la}}$ embeds into $\bW_1^{\otimes m}$ as $\gp$-module such that it maps to $V^{\otimes m} = \mathsf{hd} \, \bW_1^{\otimes m}$ with non-zero image.
\end{proof}

\begin{prop}\label{prop:Pd}
For each $\vec{\la} = ( \la^{(1)},\ldots,\la^{(\ell)}) \in \Par_{m,\ell}$, we have an inclusion
$$\P_{\vec{\la}} \subset \P_{((\emptyset)^{\ell-1},\|\vec{\la}\|)}$$
as $\gp$-modules. Moreover, $\P_{\vec{\la}}$ defines a $\gp$-submodule of $\bW_1^{\otimes m}$ that has non-zero image in $V^{\otimes m} = \mathsf{hd} \, \bW_1^{\otimes m}$ as $\gp$-module.
\end{prop}

\begin{proof}
Let us transport the functor $\mathscr D$ from the category of $\gp^{\sharp}$-modules to the category of $\gp$-modules by $\Phi$. By Corollary \ref{cor:DP}, we have $\P_{\vec{\la}} \subset \mathscr D ( \P_{\vec{\la}} )$, and $\mathscr D ( \P_{\vec{\la}} )$ surjects onto $\mathscr D ( D_{\vec{\la}} ) = D_{((\emptyset)^{\ell-1},\|\vec{\la}\|)}$. By (\ref{eqn:algDem}) and $\g^{\sharp} = (\mathfrak{gl}(n\ell) + \gp^{\sharp})$, we find that $\mathscr D ( \P_{\vec{\la}} )$ is generated by $v_{((\emptyset)^{\ell-1},\|\vec{\la}\|)^{\mathtt{tot}}}$ as $\g$-modules. In particular, we have a surjection
\begin{equation}
\mathsf{q}^{\sum_{i=1}^{\ell}(i-\ell)|\la^{(i)}|} \P_{((\emptyset)^{\ell-1},\|\vec{\la}\|)} \longrightarrow \!\!\!\!\! \rightarrow \mathscr D ( \P_{\vec{\la}} ),\label{eqn:PDPsurj}
\end{equation}
of graded $\g$-modules, where the degree shift is counted by the grading on $\mathfrak{gl}(n\ell) \subset \g$. We have an inclusion
\begin{equation}
\bW_1 \subset \C^n [z] \oplus z^{-1}\C^n [z] \oplus z^{-2}\C^n [z] \oplus \cdots \oplus z^{1-\ell}\C^n [z]\label{eqn:amb}
\end{equation}
of $\gp$-modules. We denote the RHS of (\ref{eqn:amb}) by $\bX_1$. Here, $\bX_1$ acquires the structure of a graded $\g$-module that prolongs the $\gp$-module structure. We have an induced map
$$\mathscr D ( \P_{\vec{\la}} ) \rightarrow \mathscr D ( \bW_1^{\otimes m} ) \rightarrow \bX_1^{\otimes m},$$
that is injective when restricted to $\P_{\vec{\la}}$. Here we know that $V_{\vec{\la}} \subset V^{\otimes m}$ has $|\la^{(i)}|$-contributions from the $i$-th $\C^{n}$ ($1 \le i \le \ell$) with respect to the pure tensors by (\ref{eqn:preind}). Thus, we twist the pure tensor contribution of the $i$-th $\C^{n}$ into $\C^n z^{(1-i)}$ to obtain a $\gp$-module map
$$\mathsf{q}^{\sum_{i=1}^{\ell}(1-i)|\la^{(i)}|} \P_{\vec{\la}} \subset \mathsf{q}^{\sum_{i=1}^{\ell}(1-i)|\la^{(i)}|} \mathscr D ( \P_{\vec{\la}} )\longrightarrow \bX_1^{\otimes m}.$$
This yields a map
$$\psi :\mathsf{q}^{m(1-\ell)}\P_{((\emptyset)^{\ell-1},\|\vec{\la}\|)} \longrightarrow \bX_1^{\otimes m}.$$
Since $m(1-\ell)$ is the lowest grading offered by $\bX_1^{\otimes m}$, Lemma \ref{lem:Pd} implies that $\psi$ must be injective. From this, we derive that (\ref{eqn:PDPsurj}) is injective, and hence is an isomorphism. Applying Corollary \ref{cor:DP} again, we obtain
$$\mathsf{q}^{\sum_{i=1}^{\ell}(1-i)|\la^{(i)}|} \P_{\vec{\la}} \hookrightarrow \mathsf{q}^{\sum_{i=1}^{\ell}(1-i)|\la^{(i)}|} \mathscr D ( \P_{\vec{\la}} ) \cong \mathsf{q}^{m(1-\ell)}\P_{((\emptyset)^{\ell-1},\|\vec{\la}\|)}.$$
Examining the construction and twisting back the degrees of the each tensor factor of the generating vector, we obtain an inclusion
$$\P_{\vec{\la}} \subset \mathsf{q} ^{\sum_{j=1}^{\ell} (i-\ell)|\la^{(i)}|}\P_{((\emptyset)^{\ell-1},\|\vec{\la}\|)} \subset \bX_1^{\otimes m}$$
such that $\P_{\vec{\la}}$ lands on $\bW_1^{\otimes m}$ and $\mathsf{hd} \, \P_{\vec{\la}} = V_{\vec{\la}} \subset \mathsf{hd} \, \bW_1^{\otimes m}$ as required.
\end{proof}

\begin{prop}\label{prop:Wproj}
The module $\bW_1^{\otimes m}$ is projective in $\gp \mathchar`-\mathsf{gmod}_m$.
\end{prop}

\begin{proof}
We have $\bW_1^{\otimes m} \in \gp \mathchar`-\mathsf{gmod}_m$ by Proposition \ref{prop:gen} and the fact that the $\h$-weights of $\bW_1$ is of the shape $\{\varepsilon_i\}_{i=1}^{n\ell}$. By Proposition \ref{prop:gen}, the projective cover $\mathbb P$ of $\bW_1^{\otimes m}$ in $\gp \mathchar`-\mathsf{gmod}_m$ is the same as that of $V^{\otimes m}$. By Theorem \ref{thm:cSW}, we have
\begin{equation}
\mathbb P \cong \bigoplus_{\vec{\la} \in \Par_{n,\ell}} \mathbb P_{\vec{\la}} \boxtimes L_{\vec{\la}}.\label{eqn:Pdecomp}
\end{equation}
By Lemma \ref{lem:Pd} and Proposition \ref{prop:Pd}, we find that the covering (surjective) map from $\P$ to $\bW_1^{\otimes m}$ sends one copy of $\P_{\vec{\la}}$ from (\ref{eqn:Pdecomp}) into $\bW_1^{\otimes m}$ for each $\vec{\la} \in \Par_{m,\ell}$. In view of the $W$-action that exists on the both of (\ref{eqn:Pdecomp}) and $\bW_1^{\otimes m}$, we find that the covering map from $\P$ to $\bW_1^{\otimes m}$ must be a $W$-module map whose structure commutes with the action of $\gp$. Thus, $\P_{\vec{\la}} \subset \bW_1^{\otimes m}$ implies $\P_{\vec{\la}} \boxtimes L_{\vec{\la}} \subset \bW_1^{\otimes m}$, and it forces the image to be the direct summand. Therefore, we conclude $\P \cong \bW_1^{\otimes m}$ as required.
\end{proof}

\begin{thm}\label{thm:cateq}
We have an equivalence of categories
$$\mathsf{SW} : \gp \mathchar`-\mathsf{gmod}_m \ni M \mapsto \mathrm{Hom}_{\g} (\bW_1^{\otimes m}, M ) \in A \mathchar`-\mathsf{gmod}$$
with its quasi-inverse
$$\mathsf{WS} : A \mathchar`-\mathsf{gmod} \ni N \mapsto \bW_1^{\otimes m} \otimes_{A} N \in \gp \mathchar`-\mathsf{gmod}_m.$$
\end{thm}

\begin{rem}
Theorem \ref{thm:cateq} is an extension of Feigin-Khoroshkhin-Makedonskyi \cite[Theorem C]{FKM19} to the case of wreath product. Note that Theorem \ref{thm:cateq} extends to the case of paraholic subalgebra arising from
$$\bigoplus_{i=1}^{\ell} \left( \C^{n_i} \boxtimes \chi^{i-1} \right) \hskip 10mm \text{with} \hskip 10mm \min \{ n_i \}_i \ge m$$
in place of $V$ with the same proof. Here we did not state in this form in order to keep Corollary \ref{cor:WSdual}, that incorporates $\theta$.
\end{rem}

\begin{proof}
By Proposition \ref{prop:Wproj} and Theorem \ref{thm:cSW}, we find that $\bW_1^{\otimes m}$ surjects onto all the simple graded $\gp$-modules (up to grading shift). In other words, $\bW_1^{\otimes m}$ is a progenerator of $\gp \mathchar`-\mathsf{gmod}_m$ (in the graded sense).

Thus, the equivalence of $\mathsf{SW}$ follows if we have
\begin{equation}
A \cong \mathrm{End}_{\gp} ( \bW_1^{\otimes m} )\label{eqn:Acent}
\end{equation}
by the Morita theory. We have $A \subset \mathrm{End}_{\gp} ( \bW_1^{\otimes m} )$ by enhancing the $A^{\flat}$-action by the $\Gamma^m$-action on $\bW_1^{\otimes m}$. In view of Lemma \ref{lem:multcount}, we have
\begin{align*}
\gdim \, \mathrm{End}_{\gp} ( \bW_1^{\otimes m} ) & = \sum_{\vec{\la} \in \Par_{m,\ell}} \gdim \,\mathrm{Hom}_{\mathfrak{gl}(n)^{\oplus \ell}} ( V_{\vec{\la}} \boxtimes L_{\vec{\la}}, \bW_1^{\otimes m} )\\
& = \sum_{\vec{\la} \in \Par_{m,\ell}} \frac{1}{(1-q)^m}\dim \, \mathrm{Hom}_{\mathfrak{gl}(n)^{\oplus \ell}} ( V_{\vec{\la}} \boxtimes L_{\vec{\la}}, V ^{\otimes m} ),
\end{align*}
where the second equality follows from
$$\gch \, \bW_1 = \frac{1}{1-q} \gch \, V.$$
This implies
$$\gdim \, \mathrm{End}_{\gp} ( \bW_1^{\otimes m} ) = \frac{|W|}{(1-q)^n}= \gdim \, A,$$
from Theorem \ref{thm:cSW} and Schur's lemma, that yields (\ref{eqn:Acent}).

We show that $\mathsf{WS}$ is an equivalence. The module $\bW_1^{\otimes m}$ is projective $A$-module by Lemma \ref{lem:Wproj}. In view of Theorem \ref{thm:cSW}, it is a progenerator of $A\mathchar`-\mathsf{gmod}$. Thus, the functor $\mathsf{WS}$ defines an exact functor to $\gp\mathchar`-\mathsf{gmod}_m$. Since $\mathsf{WS} (A) = \bW_1^{\otimes m}$ is the progenerator of $\gp\mathchar`-\mathsf{gmod}_m$, we find that $\mathsf{WS}$ sends an indecomposable projective to an indecomposable projective, and is fully faithful on projectives by
$$\mathrm{end}_{\gp\mathchar`-\mathsf{gmod}_m} ( \mathsf{WS} ( A )) = \mathrm{end}_{\gp\mathchar`-\mathsf{gmod}_m} ( \bW_1^{\otimes m} ) \cong A.$$
Thus, $\mathsf{WS}$ induces a category equivalence. We have
$$\mathsf{WS} \circ \mathsf{SW} \cong \mathrm{id} \hskip 5mm \text{and} \hskip 5mm \mathsf{SW} \circ \mathsf{WS}\cong \mathrm{id}$$
since it holds for indecomposable projectives.
\end{proof}

\begin{cor}\label{cor:WSsimple}
Keep the setting of Theorem \ref{thm:cateq}. We have
$$\mathsf{SW} ( V_{\vec{\la}} ) \cong L_{\vec{\la}} \hskip 5mm \text{for each} \hskip 5mm \vec{\la} \in \Par_{m,\ell}.$$
\end{cor}

\begin{proof}
We have
$$\mathsf{SW} (V_{\vec{\la}}) = \mathrm{hom}_{\g\mathchar`-\mathsf{gmod}_m} ( \bW_1^{\otimes m}, V_{\vec{\la}}) = \mathrm{Hom}_{\mathfrak{gl}(n)^{\oplus m}} ( V^{\otimes m}, V_{\vec{\la}}) \cong L_{\vec{\la}}$$
by Proposition \ref{prop:gen} and Theorem \ref{thm:cSW}.
\end{proof}

\begin{cor}\label{cor:WSdual}
Keep the setting of Theorem \ref{thm:cateq}. We have
$$\mathsf{SW} ( M^{\star} ) \cong M^{\vee} \hskip 5mm \text{and} \hskip 5mm \mathsf{WS} ( N^{\vee} ) \cong N^{\star}$$
for each finite-dimensional $M \in \gp \mathchar`-\mathsf{gmod}_m$ and $N \in A\mathchar`-\mathsf{gmod}$.	
\end{cor}

\begin{proof}
Compare Corollary \ref{cor:WSsimple} with Lemma \ref{lem:gsimple} and Lemma \ref{lem:Asimple}.
\end{proof}

\section{Kostka polynomials attached to limit symbols}\label{sec:Kostka}

Keep the setting of the previous section. We set
$$W_{\vec{\la}} := \mathsf{SW} ( \Phi ( D_{\vec{\la}} ) ), \hskip 5mm \bW_{\vec{\la}^{\vee}}^{\flat} := \mathsf{SW} ( \Phi ( \mathbb U_{\vec{\la}}^{\theta} ) ), \hskip 5mm W_{\vec{\la}^{\vee}}^{\flat} := \mathsf{SW} ( \Phi ( U_{\vec{\la}}^{\theta} ) ) \in A \mathchar`-\mathsf{gmod}$$
for each $\vec{\la} \in \Par_{m,\ell}$.

\begin{prop}\label{prop:wt-est}
Let $\la \in \Z^{n\ell}$. We expand the non-symmetric Macdonald polynomial borrowed from Theorem \ref{thm:San} $($or \rm{\cite{Che95}}$)$ as
$$E_{\la} ( x, q, t ) = \sum_{\mu \in \Z^{n\ell}} a_{\la,\mu} ( q,t ) x^{\mu} \hskip 5mm \text{where} \hskip 5mm a_{\la,\mu} (q,t) \in \C (q,t).$$
Then, we have $a_{\la,\mu} (q,0) \neq 0$ if $a_{\la,\mu}(q,t) \neq 0$. In this case, we also have
$$a_{\gamma,\mu} (q,0) \neq 0 \hskip 10mm \gamma \in \Sym_{n\ell} \la \hskip 5mm\text{such that} \hskip 5mm \lambda - \gamma  \in \Z_+^{n\ell}.$$
\end{prop}

\begin{proof}
We borrow convention from \cite{HHL08}. For each $1 \le i < n \ell$, let $s_i$ denote the operator that swaps $x_i$ with $x_{i+1}$, and $\la_i$ with $\la_{i+1}$ for $\la \in \Z^{n\ell}$, and let $T_i$ denote the Hecke operator \cite[(7)]{HHL08}. We have
\begin{equation}
(T_i - a(q,t)) x^{\la} \in \bigoplus_{j=0}^{\la_i-\la_{i+1}} \C(q,t)x^{\la - j \al_i} \hskip 3mm \la \in \Z^{n\ell}\hskip 3mm \text{such that} \hskip 3mm \la_i > \la_{i+1}\label{eqn:Tspread}
\end{equation}
by inspection. We have
\begin{equation}
E_{s_i(\la)} (x, q,t) = (T_i - a(q,t))  E_{\la} (x, q,t) \hskip 5mm \la \in \Z^{n\ell} \hskip 3mm \text{s.t.} \hskip 3mm \la_i > \la_{i+1}\label{eqn:Erec}
\end{equation}
for some $a(q,t) \in \C(q,t)$ with $a(q,0) = 1$ by \cite[(8)]{HHL08}. The $t=0$ specialization of (\ref{eqn:Erec}) yields
\begin{equation}
E_{s_i(\la)} (x, q,0) = \frac{x_i}{x_i - x_{i+1}} (1 - s_i) E_{\la} (x, q,0) \hskip 3mm \la \in \Z^{n\ell} \hskip 3mm \text{s.t.} \hskip 3mm \la_i > \la_{i+1}\label{eqn:nDem}
\end{equation}
by \cite[\S 4.1]{Ion03}. In addition, we have an operator $\Psi$ in \cite[(9)]{HHL08} that swaps non-symmetric Macdonald polynomials (\cite[(10)]{HHL08}). Note that $\Psi$ does not involve the $t$-parameter. In view of Theorem \ref{thm:San} 2), we know that (\ref{eqn:nDem}) is the numerical counterpart of the Demazure functor $\mathscr D_i$ that yields the Demazure module $D_{s_i ( \la )}$ when applied to the Demazure module $D_{\la}$. In particular, we have
$$D_{\la} \subset D_{s_i ( \la )} = \mathscr D_i ( D_{\la} ) \subset \text{integrable highest weight module}.$$
Let $\mathfrak{sl}(2,i)$ be the Lie subalgebra of $\mathfrak{gl}(n\ell)$ spanned by
$$\{E_{i,i+1},E_{i+1,i},(E_{ii}-E_{i+1,i+1})\}.$$
The output of $\mathscr D_i$ is stable under the action of $\mathfrak{sl}(2,i)$ by its definition \cite[\S 2.1]{Jos85}. In view of \cite[\S 2.3]{Jos85} and $\dim \, D_{s_i ( \la )} < \infty$, every non-zero vector of $\mathscr D_i ( D_{\la} )$ generates a $\mathfrak{sl}(2,i)$-module whose highest weight is a $\h$-weight of $D_{\la}$. In particular, we have
$$\mathrm{Hom}_\h ( \C_{\mu}, D_{s_i(\la)}) \neq 0 \hskip 5mm \mu \in \Z^{n\ell}$$
if and only if there exists $\nu \in \Z^{n\ell}$ such that $\nu_i \ge \nu_{i+1}$, $\mathrm{Hom}_\h ( \C_{\nu}, D_{\la}) \neq 0$, and
$$\mu \in \{\nu,\nu-\al_i,\ldots,\nu-(\nu_i-\nu_{i+1})\al_i\}.$$
Comparing this with (\ref{eqn:Tspread}), we conclude the assertions by induction on the recursive definition of non-symmetric Macdonald polynomials summarized as (\ref{eqn:Erec}) and \cite[(9)]{HHL08}.
\end{proof}

\begin{cor}\label{cor:wt-est}
Keep the setting of Proposition \ref{prop:wt-est}. Let $\vec{\la},\vec{\mu} \in \Par_{m,\ell}$ and $\sigma \in \Sym$ be such that $a_{-\sigma \vec{\la}^{\mathtt{tot}},-\vec{\mu}^{\mathtt{tot}}} (q,0) \neq 0$. We have $\vec{\la}^* \unrhd \vec{\mu}^*$.	
\end{cor}

\begin{proof}
Taking Lemma \ref{lem:order} into account, this is a consequence of the triangularity of the non-symmetric Macdonald polynomials (\cite[Definition 2.2.1 i)]{HHL08}).
\end{proof}

\begin{lem}\label{lem:mult}
Let $\vec{\la},\vec{\mu} \in \Par_{m,\ell}$. We have
$$\bigl(\left. \gdim\, R_{\vec{\la}^{\vee}} \right|_{q \mapsto q^{\ell}} \bigr)^{-1} \cdot [\bW_{\vec{\la}}^{\flat} : V_{\vec{\mu}}]_q = [W_{\vec{\la}}^{\flat} : V_{\vec{\mu}}] _q \neq \delta_{\vec{\la},\vec{\mu}} \hskip 5mm \text{implies}\hskip 5mm \vec{\la}^{\vee} \rhd \vec{\mu}^{\vee}.$$
\end{lem}

\begin{proof}
The equality is the comparison of Theorem \ref{thm:DU} 6) and \cite[Corollary 5.5]{FKM23}. By Theorem \ref{thm:DU} 1), we have
$$[U_{\vec{\la}}:V_{\vec{\mu}^*}]_q \neq \delta_{\vec{\la},\vec{\mu}} \hskip 3mm \Rightarrow \hskip 3mm \vec{\la}^{*} \rhd \vec{\mu}^{*}.$$
Applying twisting by $\theta$ and $\mathsf{SW} \circ \Psi$ yields the desired implication.
\end{proof}

\begin{lem}\label{lem:ext}
We have $\mathrm{ext}^i_{A \mathchar`-\mathsf{gmod}} ( \bW_{\vec{\la}}^{\flat}, W_{\vec{\mu}}^{\star}) = \C^{\delta_{i0}\delta_{\vec{\la},\vec{\mu}^{\vee}}}$.
\end{lem}

\begin{proof}
This is a restatement of Theorem \ref{thm:DU} 3) in view of (\ref{eqn:ungradedExt}).
\end{proof}

\begin{prop}\label{prop:socdeg}
For each $\vec{\la} \in \Par_{m,\ell}$, the graded $A$-module $W_{\vec{\la}}$ has a simple head $V_{\vec{\la}}$ and a simple socle $\mathsf{q}^{\mathsf{a} ( \vec{\la} )} V_{((1^m)(\emptyset)^{\ell-1})}$.
\end{prop}

\begin{proof}
The assertion on simple head and simple socle is Corollary \ref{cor:simple}. Thus, it remains to calculate the degree on which the simple socle sits. Let $\la_-$ and $\la_+$ be the largest and the smallest elements in $\bX^-_{\ell}$ with respect to $\lhd$ obtained as the permutation of $\vec{\la}^{\mathtt{tot}}$, respectively. We may identify $\la_+$ with a $\ell$-partition of shape $(((\la_+)_n,(\la_+)_{n-1},\ldots,(\la_+)_1))(\emptyset)^{\ell-1})$ by $m \le n$.

We have $D_{\vec{\la}} \subset D_{\la_-}$. In view of \cite[1.5.2]{CL06} and \cite[p.~243 (6.5)(ii)]{Mac95}, the socle $V_{((1^m)(\emptyset)^{\ell-1})}$ sits in degree $\mathsf{u} ( \la_+ )$ as $\g^{\sharp}$-modules by
$$K_{\la,(m)} ( t ) =  t^{\mathsf{n}(\la) - \mathsf{n} ( m )} + \text{(lower terms)} = t^{\mathsf{u}(\la') - \mathsf{u} ( 1^m )} + \text{(lower terms)}$$
 for a partition $\mu$ and $\mathsf{u} (1^m)=0$. Here we have twisted the grading in accordance with the isomorphism $\Phi$, that raises the $z$-grading by $\ell$, together with the grading shifts that assigns $E_{ij} \in \gp^{\sharp}$ ($1 \le i, j \le n \ell$) with
 $$\lceil\frac{j}{n}\rceil - \lceil\frac{i}{n}\rceil \neq 0.$$
In view of the commutation relation of bases of $\gp^{\sharp}$ (see also Remark \ref{rem:grading}), we find that these two contributions are independent and the total grading shift is well-defined (cf. Remark \ref{rem:multigr}). The $z$-grading twist contributes by $q^{\ell\mathsf{u}(\la)}$. Since the grading twists needed to send $\vec{\la}$ to $\la_+$ is given as
$$\sum_{k=1}^{\ell} ( k-1 ) |\la^{(k)}|,$$
we conclude that the socle $V_{((m)(\emptyset)^{\ell-1})}$ must be in degree $\mathsf{a} ( \vec{\la} )$ as required.
\end{proof}

\begin{cor}\label{cor:min}
For each $\vec{\la} \in \Par_{m,\ell}$, the graded module $\mathsf{q}^{\mathsf{a} ( \vec{\la})} ( W_{\vec{\la}} )^{\vee}$ is a quotient of the graded $A$-module $\C [X]$.	In addition, we have
\begin{equation}
\dim \, \Hom_{W} ( L_{\la}^{\vee}, \C [X]_{i} ) = \delta_{i,\mathsf{a} ( \vec{\la} )} \hskip 5mm \text{when} \hskip 5mm i \le \mathsf{a} ( \vec{\la}).\label{eqn:minimal}
\end{equation}
\end{cor}

\begin{proof}
Thanks to Proposition \ref{prop:socdeg} and our convention, $\mathsf{q}^{\mathsf{a} ( \vec{\la})} ( W_{\vec{\la}} ) ^{\vee}$ is a graded $A$-module generated by the trivial representation of $W$ sitting at degree zero. Thus, it is a quotient of $\C[X]$. We have $\mathsf{soc}\,\mathsf{q}^{\mathsf{a} ( \vec{\la})} ( W_{\vec{\la}})^{\vee} \cong \mathsf{q}^{\mathsf{a}( \vec{\la})} L_{\vec{\la}^{\vee}}$.

Note that the minimal degree in which $\Sym$-module (\ref{eqn:preind}) appears in $\C [X]$ is $\sum_{k=1}^{\ell} \mathsf{u} ( \la^{(k)})$ (\cite[Theorem 1]{ATY}). By \cite[Theorem 2 3)]{ATY}, we conclude that $\mathsf{a} ( \vec{\la} )$ is the minimal degree on which $\C [X]$ carries $L_{\vec{\la}^{\vee}}$ uniquely under the convention (\ref{eqn:G-action}). Thus, we conclude the second assertion.
\end{proof}

\begin{thm}\label{thm:id}
We have
\begin{equation}
[W_{\vec{\mu}} : V_{\vec{\la}}]_q = \delta_{\vec{\la},\vec{\mu}}  \hskip 3mm \text{and} \hskip 3mm [W_{\vec{\mu}^{\vee}}^{\flat} : V_{\vec{\la}^{\vee}}]_q = \delta_{\vec{\la}^{\vee},\vec{\mu}^{\vee}}\label{eqn:Kpm}
\end{equation}
when $\mathsf{a} (\vec{\la}) \ge \mathsf{a} (\vec{\mu})$. It follows that the multiplicities in $(\ref{eqn:Kpm})$, without the condition on the $\mathsf{a}$-values, represent Kostka polynomials $K^-$ and $K^+$ attached to the limit symbols $(\rm{\cite{Sho04}})$ up to the componentwise conjugation on the labels.
\end{thm}

\begin{proof}
The first multiplicity count follows from Corollary \ref{cor:min}. In view of Corollary \ref{cor:gch} and Corollary \ref{cor:wt-est}, we find
$$[U_{\vec{\mu}} : V_{\vec{\la}^*}]_q \neq 0 \Rightarrow [D_{\vec{\mu}^{*}} : V_{\vec{\la}^{*}}]_q \neq 0.$$
Thus, we have $[U_{\vec{\mu}} : V_{\vec{\la}^*}]_q = \delta_{\vec{\mu}^*, \vec{\la}^*}$ when $\mathsf{a} ( \vec{\mu}^* ) \ge \mathsf{a} ( \vec{\la}^* )$ by the first multiplicity count. Twisting by $\theta$, the second multiplicity estimate follows. Since the output of the Lusztig-Shoji algorithm for symbols only depends on the values of $\mathsf{a}$-functions, we conclude the third assertion from (\ref{eqn:Kpm}), Lemma \ref{lem:ext}, and Theorem \ref{thm:homLSalg} by \cite[\S 1.5, \S 1.6]{Sho02} and \cite[\S 3.1]{Sho04}.
\end{proof}

\begin{rem}\label{rem:inc}
When $\ell = 2$, we have $W_{\vec{\la}} = W_{\vec{\la}}^{\flat}$ for each $\vec{\la} \in \Par_{m,2}$ by \cite[Theorem 5.6]{Kat17}. It yields a non-trivial isomorphism $D_{\vec{\la}} \cong U^{\theta}_{\vec{\la}}$ for each $\vec{\la} \in \Par_{m,2}$ (see \S \ref{subsec:222}).
\end{rem}

The following result is the positivity of Kostka polynomials conjectured in \cite[Conjecture 5.5]{Sho02} that is verified for $K^-$ in the case of the limit symbols \cite{FI18,Hu18,Sho18}. In particular, it is new for $K^+$ and gives a new proof for $K^-$ based on \cite[Theorem 2.10]{Kat15} and Theorem \ref{thm:cateq}.

\begin{cor}\label{cor:Kpos}
For each $\vec{\la},\vec{\mu} \in \Par_{m,\ell}$, we have
$$[W_{\vec{\mu}} : V_{\vec{\la}}]_q,\hskip 3mm [W_{\vec{\mu}}^{\flat} : V_{\vec{\la}}]_q \in \Z_{\ge 0}[q].$$
\end{cor}
\begin{proof}
Immediate from Theorem \ref{thm:id} since $W_{\vec{\mu}}$ and $W_{\vec{\mu}}^{\flat}$ are finite-dimensional graded $A$-modules by Theorem \ref{thm:DU}.
\end{proof}

\begin{rem}\label{rem:multigr}
In \cite{Sho18}, the multivariable version of Kostka polynomials introduced in \cite{FI18} is used effectively. We can see these variables since the modules $\Phi (D_{\vec{\la}})$ and $\Phi (U_{\vec{\la}}^{\theta})$ ($\vec{\la} \in \Par_{m,\ell}$) are in fact multigraded by Remark \ref{rem:grading} and the fact that the defining equations of $D^{J}_{\vec{\la}}$ and $U_{\vec{\la}}^J$ in \cite[\S 4.1]{FKM23} can be understood to be homogeneous with respect to $\mathtt{deg}$. It follows that $W_{\vec{\la}}$ and $W_{\vec{\la}}^{\flat}$ admit gradings induced from $\mathtt{deg}$ as vector spaces. A grading $\mathtt{deg}$ on $A$ that is compatible with $\mathtt{deg}$ on $W_{\vec{\la}}$ and $W_{\vec{\la}}^{\flat}$ cannot have a single grading on
$$\C [X]_1 \setminus \{0\} = \Bigl( \bigoplus_{j=1}^m \C X_j \Bigr) \setminus \{0\}\subset A$$
since the multiplication by $\C [X]_1$, that is irreducible as a $W$-module, does not preserve the $\mathtt{deg}$-grading already when $(m,\ell) = (2,2)$.\\
Instead, we must equip the $\mathtt{deg}$-grading on $A$ by using the idempotents
$$1 = \sum_{\vec{\la}\in \Par_{m,\ell}} 1_{\vec{\la}} \in \C [W] \hskip 5mm \textrm{such that} \hskip 5mm \mathrm{End} ( L_{\vec{\la}} )^{\oplus \delta_{\vec{\la},\vec{\mu}}} \cong 1_{\vec{\la}}\C [W]1_{\vec{\mu}}$$
for each $\vec{\la}, \vec{\mu} \in \Par_{m,\ell}$ and set
$$\mathtt{deg} \, ( 1_{\vec{\la}}\C [X]_1 1_{\vec{\mu}} \setminus \{0\}) := \mathtt{e}_i \hskip 10mm \text{if} \hskip 10mm i=j-1<j \text{ or } i=\ell,j=1$$
when we have have $\la^{(k)} = \mu^{(k)}$ for $k \neq i, j$, $\la^{(i)}$ is obtained by adding one box in $\mu^{(i)}$, and $\la^{(j)}$ is obtained by removing one box in $\mu^{(j)}$. This answers a question in \cite[Remark 2.16]{Sho18}.
\end{rem}

\begin{cor}[Shoji's conjecture \cite{Sho04} \S 3.13]\label{cor:Shoji}
For each $\vec{\la},\vec{\mu} \in \Par_{m,\ell}$, $K^-_{\vec{\la}, \vec{\mu}}(q)$ counts the graded occurence of $L_{\vec{\la}^{\vee}}$ in a quotient module $\mathsf{q}^{\mathsf{a}(\vec{\mu})} ( W_{\vec{\mu}} )^{\vee}$ of $\C[X]$.
\end{cor}

\begin{proof}
Combine Theorem \ref{thm:id} and Corollary \ref{cor:min}.
\end{proof}

\begin{cor}\label{cor:minchar}
For each $\vec{\la} \in \Par_{m,\ell}$, every graded $A$-module quotient of $\C[X]$ that contains $L_{\vec{\la}^{\vee}}$ as nongraded $W$-modules admits a surjection to $\mathsf{q}^{\mathsf{a}(\vec{\la})} ( W_{\vec{\la}} )^{\vee}$.
\end{cor}

\begin{proof}
Consider the subgroup
$$W_0 := \prod_{k=1}^{\ell} \prod_{j=1}^n \left( \Sym_{\la_{j}^{(k)}} \ltimes \Gamma^{|\la_{j}^{(k)}|} \right) \subset W.$$
For each $1 \le j \le n, 1 \le k \le \ell$, we set
$$\la_{<j}^{(k)} := \left( \sum_{s=1}^{k-1} |\la^{(s)}| \right) + \sum_{i=1}^{j-1} \la_{i}^{(k)} \hskip 5mm \text{and} \hskip 5mm \la_{\le j}^{(k)} := \left( \sum_{s=1}^{k-1} |\la^{(s)}| \right) + \sum_{i=1}^{j} \la_{i}^{(k)}.$$
Then, $W_0$ preserves each $\C[X_{\la_{<j}^{(k)}+1},\ldots,X_{\la_{\le j}^{(k)}}] \subset \C[X]$. In addition, the $W$-submodule $\mathsf{q}^{\mathsf{a}(\vec{\la})} L_{\vec{\la}}^{\vee} \subset \C[X]$ contains a one-dimensional space $L$ spanned by
$$\prod_{k=1}^{\ell}\prod_{j=1}^n \Bigl( \prod_{i = \la_{<j}^{(k)} + 1}^{\la_{\le j}^{(k)}} X_i^{k-1} \Bigr) \cdot \left( \prod_{\la_{< j}^{(k)} < i < i' \le \la_{\le j}^{(k)}} ( X_i^{\ell} - X_{i'}^{\ell} ) \right).$$
This is the minimal degree realization of $L$ as $W_0$-module inside $\C [X]$. Thanks to \cite[Theorem 3.1]{Sta77}, the module $\mathrm{Hom}_{W_0} (L,\C[X])$ is a free $\C[X]^{W_0}$-module of rank one. We dualize the situation using the perfect pairing between $\C [X]$ and $\C [\partial_1,\ldots,\partial_m]$ (with $\partial_i = \frac{\partial}{\partial X_i}$ and $\deg \, \partial_i = -1$ for $1 \le i \le m$) given as
$$\C [\partial_1,\ldots,\partial_m] \times \C[X] \ni ( P(\partial), Q(X) ) \mapsto \left. \left( P(\partial) Q(X) \right) \right|_{X_i = 0} \in \C.$$
Applying a $W_0$-invariant differential operator with constant coefficients sends an arbitrary $L$-isotypical components of $\C [X]$ to its minimal degree realization. Here $-X_i$ acts on $\partial_j$ as a differential operator by $[X_i,\partial_j] = \delta_{ij}$ for $1 \le i,j\le m$. Since an $A$-module quotient $M$ of $\C[X]$ that contains $L_{\vec{\la}^{\vee}}$ contains $L$ as $W_0$-modules, we conclude an inclusion
$$\mathsf{q}^{-\mathsf{a}(\vec{\la})} W_{\vec{\la}} \subset \mathsf{q}^{-\mathsf{a}(\vec{\la})} M^{\vee} \subset \C [\partial_1,\ldots,\partial_m].$$
Dualizing this yields the desired surjection.
\end{proof}

\section{Examples}
We retain the settings of the previous section.
\subsection{$m=n=2$,\, $\ell=2$}\label{subsec:222}
We have
$$((1^2)(\emptyset)) \lhd ((1)(1)) \lhd ((\emptyset)(1^2)) \lhd ((2)(\emptyset)) \lhd ((\emptyset)(2)).$$
The computation of affine Demazure characters yield:
\begin{align*}
\gch \, D_{((1^2)(\emptyset))} & = \ch \, V_{((1^2)(\emptyset))} = x_1 x_2, & \gch \, D_{((1)(1))} & = \ch \, V_{((1)(1))} + \ch \, V_{((1^2)(\emptyset))}\\
\gch \, D_{((\emptyset)(1^2))} & = \ch \, V_{((\emptyset)(1^2))} +  \gch \, D_{((1)(1))}, & \gch \, D_{((2)(\emptyset))} & = \ch \, V_{((2)(\emptyset))} + q \gch \, D_{((1)(1))}
\end{align*}
and
$$\gch \, D_{((\emptyset)(2))} = \ch \, V_{((2)(\emptyset))} + \ch \, V_{((1)(1))} +  \gch \, D_{((2)(\emptyset))}.$$
Applying $\Phi$, we obtain
\begin{align*}
\gch \, \Phi (D_{((\emptyset)(1^2))})  = & \ch \, V_{((\emptyset)(1^2))} + q \ch \, V_{((1)(1))} + q^2 \ch \, V_{((1^2)(\emptyset))}\\
\gch \, \Phi ( D_{((2)(\emptyset))} )  = &\ch \, V_{((2)(\emptyset))} + q \ch \, V_{((1)(1))} + q^2 \ch \, V_{((1^2)(\emptyset))}\\
\gch \, \Phi ( D_{((\emptyset)(2))} ) =  &\ch \, V_{((\emptyset)(2))} + (q + q^3) \ch \, V_{((1)(1))} + q^2 \ch \, V_{((2)(\emptyset))} \\
& + q^2 \ch \, V_{((\emptyset)(1^2))} + q^4 \ch \, V_{((1^2)(\emptyset))}
\end{align*}
This agrees with \cite[p465 Table 1]{Sho04} (that describes the {\it modified} Kostka polynomials) up to the componentwise conjugate (explained in \S \ref{subsec:algA}):
\begin{center}
\begin{tabular}{c|ccccc}
&$((\emptyset)(2))$&$((2)(\emptyset))$&$((\emptyset)(1^2))$&$((1)(1))$&$((1^2)(\emptyset))$\\\hline
$((\emptyset)(2))$ & $1$ & & & & \\
$((2)(\emptyset))$ & $q^2$ & $1$ & & & \\
$((\emptyset)(1^2))$ & $q^2$ & & $1$ & & \\
$((1)(1))$ & $q+q^3$ & $q$ & $q$ & $1$ & \\
$((1^2)(\emptyset))$& $q^4$ & $q^2$& $q^2$ & $q$ & $1$
\end{tabular}
\end{center}
Here the missing entries are regarded as zero.

In this case, we have $\vec{\la}^{\vee} \equiv \vec{\la}$ and $W_{\vec{\la}}^{\flat} = W_{\vec{\la}}$ for every $\vec{\la} \in \Par_{2,2}$. In particular,
\begin{align*}
\gch \, \Phi ( D_{((1^2)(\emptyset))} )  = & x_1 x_2\\
\gch \, \Phi ( D_{((1)(1))} ) = & x_1 x_3 + x_1 x_4 + x_2 x_3 + x_2 x_4 + q x_1 x_2 \\
\gch \, \Phi ( D_{((\emptyset)(1^2))} ) = & x_3 x_4 + q( x_1 x_3 + x_1 x_4 + x_2 x_3 + x_2 x_4 ) + q^2 x_1 x_2.
\end{align*}
and
\begin{align*}
\gch \, \Phi ( U_{((1^2)(\emptyset))} ) = & x_3 x_4\\
\gch \, \Phi ( U_{((1)(1))} ) = & x_1 x_3 + x_1 x_4 + x_2 x_3 + x_2 x_4 + q x_3 x_4\\
\gch \, \Phi ( U_{((\emptyset)(1^2))} )  = & x_1 x_2 +  q( x_1 x_3 + x_1 x_4 + x_2 x_3 + x_2 x_4 ) + q^2 x_3 x_4
\end{align*}
are identified through $\theta$-twists (Remark \ref{rem:inc}).

\subsection{$m=n=1$,\, $\ell=3$}
We have
\begin{align*}
((1)(\emptyset)(\emptyset)) &\lhd ((\emptyset)(1)(\emptyset))\lhd ((\emptyset)(\emptyset)(1)).
\end{align*}
In this case, we have $D_{\vec{\la}} \subset D_{(001)} \cong \C^3$ for each $\vec{\la} \in \Par_{1,3}$. We have
\begin{align*}
\gch \, \Phi ( D_{((1)(\emptyset)(\emptyset))} )  = & x_1\\
\gch \, \Phi ( D_{((\emptyset)(1)(\emptyset))} ) = & x_2 + q x_1 \\
\gch \, \Phi ( D_{((\emptyset)(\emptyset)(1))} ) = & x_3 + q x_2 + q^2 x_1.
\end{align*}
It follows that
\begin{align*}
[W_{((1)(\emptyset)(\emptyset))}] & = [L_{((1)(\emptyset)(\emptyset))} ], \hskip 3mm 
[W_{((\emptyset)(1)(\emptyset))}] = [L_{((\emptyset)(1)(\emptyset))}] + q [L_{((1)(\emptyset)(\emptyset))}],\\
[W_{((\emptyset)(\emptyset)(1))}] & = [L_{((\emptyset)(\emptyset)(1))}] + q [L_{((\emptyset)(1)(\emptyset))}] + q^2 [L_{((1)(\emptyset)(\emptyset))}],
\end{align*}
where $[\bullet]$ denote the class in the Grothendieck group of $A\mathchar`-\mathsf{gmod}$ and $q$ denotes its grading shifts.

On the other hand, the modules $\{ \mathbb U_{\vec{\la}}\}_{\vec{\la} \in \Par_{1,3}}$ are obtained as the successive quotients of the inclusion relation of cyclic $\gb^{\sharp}$-modules
$$( \C \oplus z\C^3 [z] )\subset ( \C^2 \oplus z \C^3 [z] ) \subset \C^3 [z],$$
where $\C \subset \C^2 \subset \C^3$ are the $\gb_0$ submodules of the vector representation of $\mathfrak{gl} (3)$. This yields
$$\gch \, \mathbb U_{(001)} = x_1 + q\frac{x_1+x_2+x_3}{1-q}, \hskip 3mm \gch \, \mathbb U_{(010)} = x_2, \hskip 3mm \text{and} \hskip 3mm \gch \, \mathbb U_{(100)} = x_3.$$
It follows that
\begin{align*}
[W^{\flat}_{((\emptyset)(1)(\emptyset))}] & \equiv [ W^{\flat}_{((\emptyset)(\emptyset)(1))^{\vee}}] = [L_{((\emptyset)(1)(\emptyset))} ]+ q [L_{((1)(\emptyset)(\emptyset))}] + q^2 [L_{((\emptyset)(\emptyset)(1))}]\\
[W^{\flat}_{((\emptyset)(\emptyset)(1))}] & \equiv  [ W^{\flat}_{((\emptyset)(1)(\emptyset))^{\vee}}] = [L_{((\emptyset)(\emptyset)(1))}]\\
[W^{\flat}_{((1)(\emptyset)(\emptyset))}] & \equiv  [ W^{\flat}_{((1)(\emptyset)(\emptyset))^{\vee}}] = [L_{((1)(\emptyset)(\emptyset))}].
\end{align*}

In particular, we obtain the following tables:
\begin{center}
\begin{tabular}{c|ccc}
$K^-$ &$((\emptyset)(\emptyset)(1))$&$((\emptyset)(1)(\emptyset))$&$((1)(\emptyset)(\emptyset))$\\\hline
$((\emptyset)(\emptyset)(1))$ & $1$ & & \\
$((\emptyset)(1)(\emptyset))$ & $q$ & $1$ &  \\
$((1)(\emptyset)(\emptyset))$ & $q^2$ & $q$ & $1$ \\
\end{tabular}\\
\smallskip
\begin{tabular}{c|ccc}
$K^+$ &$((\emptyset)(1)(\emptyset))$&$((\emptyset)(\emptyset)(1))$&$((1)(\emptyset)(\emptyset))$\\\hline
$((\emptyset)(1)(\emptyset))$ & $1$ & &  \\
$((\emptyset)(\emptyset)(1))$ & $q^2$ & $1$  & \\
$((1)(\emptyset)(\emptyset))$ & $q$ & & $1$ \\
\end{tabular}
\end{center}
Here the missing entries are regarded as zero.

\subsection{$m=n=1$,\, $\ell > 3$}
We set $\{i\} := ((\emptyset)^{\ell-i}(1)(\emptyset)^{i-1})$ for $0\le i < \ell$. Then, we have
\begin{align*}
\{\ell\} \lhd \{\ell-1\} \lhd \cdots \lhd \{2\}\lhd \{1\}.
\end{align*}
We have
\begin{align*}
& \gch \, D_{(0^{\ell-i}10^{i-1})} = x_1 + \cdots + x_{\ell+1-i}, \hskip 5mm \text{and}\\
& \gch \, \mathbb U_{(0^{\ell-1}1)} = x_{1} + q \frac{x_1+\cdots + x_\ell}{1-q} , \hskip 3mm \gch \, \mathbb U_{(0^{\ell-i}10^{i-1})} = x_{i}, \hskip 5mm 1< i \le \ell,
\end{align*}
that implies
$$K_{\{i\},\{j\}}^- = \begin{cases} 0 & (i < j)\\ q^{i-j} & (i \ge j)\end{cases} \hskip 5mm \text{and} \hskip 5mm K_{\{i\},\{j\}}^+ = \begin{cases}1 & (i = j)\\ q & (i =\ell, j=\ell-1)\\ q^{i+1} & (i <\ell-1, j=\ell-1) \\ 0 & (else) \end{cases}.$$

\medskip

{\small
{\bf Acknowledgement:} This research is supported in part by JSPS KAKENHI Grant Number JP19H01782 and JP24K21192. The author thanks Evgeny Feigin for helpful communication.}

{\footnotesize
\bibliography{ref}

\begin{thebibliography}{10}

\bibitem{ATY}
Susumu Ariki, Tomohide Terasoma, and Hiro-Fumi Yamada.
\newblock Higher {S}pecht polynomials.
\newblock {\em Hiroshima Math. J.}, 27:177--188, 1997.

\bibitem{CL06}
Vyjayanthi Chari and Sergei Loktev.
\newblock Weyl, {D}emazure and fusion modules for the current algebra of
  {$\mathfrak{sl}_{r+1}$}.
\newblock {\em Adv. Math.}, 207(2):928--960, 2006.

\bibitem{Che95}
Ivan Cherednik.
\newblock {N}onsymmetric {M}acdonald {P}olynomials.
\newblock {\em International Mathematics Research Notices}, 2(10), 1995.

\bibitem{DP81}
Corrado {De Concini} and Claudio Procesi.
\newblock Symmetric functions, conjugacy classes and the flag variety.
\newblock {\em Invent. Math.}, 64:203--219, 1981.

\bibitem{FKM}
Evgeny Feigin, Syu Kato, and Ievgen Makedonskyi.
\newblock Representation theoretic realization of non-symmetric {M}acdonald
  polynomials at infinity.
\newblock {\em J. reine angew. Math.}, 767:181--216, 2020.
\newblock arXiv: 1703.04108.

\bibitem{FKM23}
Evgeny Feigin, Anton Khoroshkin, and Ievgen Makedonskyi.
\newblock Parahoric {L}ie algebras and parasymmetric {M}acdonald polynomials.
\newblock arXiv:2311.12673.

\bibitem{FKM19}
Evgeny Feigin, Anton Khoroshkin, and Ievgen Makedonskyi.
\newblock Duality theorems for current groups.
\newblock {\em Israel Journal of Mathematics}, 248:441--479, 2022.

\bibitem{FKMO}
Evgeny Feigin, Anton Khoroshkin, Ievgen Makedonskyi, and Daniel Orr.
\newblock Peter-{W}eyl theorem for {I}wahori groups and highest weight
  categories.
\newblock arXiv:2307.02124.

\bibitem{FMO23}
Evgeny Feigin, Ievgen Makedonskyi, and Daniel Orr.
\newblock Nonsymmetric $q$-{C}auchy identity and representations of the
  {I}wahori algebra.
\newblock {\em Kyoto J. Math.}, to appear.

\bibitem{FI18}
Michael Finkelberg and Andrei Ionov.
\newblock Kostka--{S}hoji polynomials and {L}usztig's convolution diagram.
\newblock {\em Bull. Inst. Math. Acad. Sinica (New Series)}, 13(1):31--42,
  2018.

\bibitem{Fl21}
Yuval~Z. Flicker.
\newblock Affine {S}chur duality.
\newblock {\em J. Lie theory}, 31(3):681--718, 2021.

\bibitem{Gre55}
James~A. Green.
\newblock The characters of the finite general linear groups.
\newblock {\em Trans. Amer. Math. Soc.}, 80:402--447, 1955.

\bibitem{Gup87}
R.~K. Gupta.
\newblock Characters and the {$q$}-analog of weight multiplicity.
\newblock {\em J. London Math. Soc. (2)}, 36(1):68--76, 1987.

\bibitem{Hai03}
Mark Haiman.
\newblock Combinatorics, symmetric functions, and {H}ilbert schemes.
\newblock In {\em Current developments in mathematics, 2002}, pages 39--111.
  Int. Press, Somerville, MA, 2003.

\bibitem{Har77}
Robin {Hartshorne}.
\newblock {\em {Algebraic geometry.}}, volume~52 of {\em Graduate Text in
  Mathematics}.
\newblock Springer- Verlag, 1977.

\bibitem{HS77}
Ryoshi Hotta and T.~A. Springer.
\newblock A specialization theorem for certain {W}eyl group representations and
  an application to the {G}reen polynomials of unitary groups.
\newblock {\em Invent. Math.}, 41:113--127, 1977.

\bibitem{Hu18}
Yue Hu.
\newblock Higher cohomology vanishing of line bundles on generalized {S}pringer
  resolution.
\newblock {\em Funct. Anal. Appl.}, 52:214--223, 2018.

\bibitem{Ion03}
Bogdan Ion.
\newblock Nonsymmetric {M}acdonald polynomials and {D}emazure characters.
\newblock {\em Duke Math. J.}, 116(2):299--318, 2003.

\bibitem{HHL08}
M.~Haiman J.~Haglund and N.~Loehr.
\newblock A combinatorial formula for nonsymmetric {M}acdonald polynomials.
\newblock {\em Amer. J. Math.}, 130(2):359--383, 2008.

\bibitem{Jan03}
J.~C. Jantzen.
\newblock {\em Representations of Algebraic Groups}.
\newblock Amer Mathematical Society, second edition edition, 2003.

\bibitem{Jos85}
Anthony Joseph.
\newblock On the {D}emazure character formula.
\newblock {\em Ann. Sci. \'Ecole Norm. Sup. (4)}, 18(3):389--419, 1985.

\bibitem{Kac}
Victor~G. Kac.
\newblock {\em Infinite-dimensional {L}ie algebras}.
\newblock Cambridge University Press, Cambridge, third edition, 1990.

\bibitem{Kat15}
Syu Kato.
\newblock A homological study of {G}reen polynomials.
\newblock {\em Ann. Sci. \'Ec. Norm. Sup\'er. (4)}, 48(5):1035--1074, 2015.

\bibitem{Kat17}
Syu Kato.
\newblock An algebraic study of extension algebras.
\newblock {\em Amer. J. Math.}, 139(3):567--615, 2017.
\newblock arXiv:1207.4640.

\bibitem{KP79}
Hanspeter Kraft and Claudio Procesi.
\newblock Closures of conjugacy classes of matrices.
\newblock {\em Invent. Math.}, 53:227--247, 1979.

\bibitem{Kum02}
Shrawan Kumar.
\newblock {\em Kac-{M}oody groups, their flag varieties and representation
  theory}, volume 204 of {\em Progress in Mathematics}.
\newblock Birkh\"auser Boston, Inc., Boston, MA, 2002.

\bibitem{Lus84}
G.~Lusztig.
\newblock Intersection cohomology complexes on a reductive group.
\newblock {\em Invent. Math.}, 75(2):205--272, 1984.

\bibitem{Lus81a}
George Lusztig.
\newblock Singularities, character formulas, and a {$q$}-analog of weight
  multiplicities.
\newblock In {\em Analysis and topology on singular spaces, {II}, {III}
  ({L}uminy, 1981)}, volume 101 of {\em Ast\'erisque}, pages 208--229. Soc.
  Math. France, Paris, 1983.

\bibitem{Lus90b}
George Lusztig.
\newblock Green functions and character sheaves.
\newblock {\em Ann. of Math. (2)}, 131:355--408, 1990.

\bibitem{Mac95}
I.~G. Macdonald.
\newblock {\em Symmetric functions and {H}all polynomials}.
\newblock Oxford Mathematical Monographs. The Clarendon Press, Oxford
  University Press, New York, second edition, 1995.
\newblock With contributions by A. Zelevinsky, Oxford Science Publications.

\bibitem{NY97}
Atsushi Nakayashiki and Yasuhiko Yamada.
\newblock Kostka polynomials and energy functions in solvable lattice models.
\newblock {\em Selecta Math. (N.S.)}, 3:547--599, 1997.

\bibitem{OS22}
Daniel Orr and Mark Shimozono.
\newblock Quiver {H}all--{L}ittlewood functions and {K}ostka--{S}hoji
  polynomials.
\newblock {\em Pacific J. Math.}, 319(2):397--437, 2022.

\bibitem{Reg86}
Amitai Regev.
\newblock The representations of wreath products via double centralizing
  theorems.
\newblock {\em J. Algebra}, 102:423--443, 1986.

\bibitem{San00}
Yasmine~B. Sanderson.
\newblock On the connection between {M}acdonald polynomials and {D}emazure
  characters.
\newblock {\em J. Algebraic Combin.}, 11(3):289--275, 2000.

\bibitem{Sho83}
Toshiaki Shoji.
\newblock On the {G}reen polynomials of classical groups.
\newblock {\em Invent. Math.}, 74:239--263, 1983.

\bibitem{Sho02}
Toshiaki Shoji.
\newblock Green functions associated to complex reflection groups.
\newblock {\em J. Algebra}, 245(2):650--694, 2001.

\bibitem{Sho04}
Toshiaki Shoji.
\newblock Green functions attached to limit symbols.
\newblock {\em Adv. Stud. Pure Math.}, 40:443--467, 2004.

\bibitem{Sho18}
Toshiaki Shoji.
\newblock Kostka functions associated to complex reflection groups and a
  conjecture of {F}inkelberg-{I}onov.
\newblock {\em Science China Mathematics}, 61:353--384, 2018.

\bibitem{Spe32}
Wilhelm Specht.
\newblock Eine verallgemeinerung der symmetrischen gruppe.
\newblock {\em Schriften Math. Sem. Berlin}, 1:1--32, 1932.

\bibitem{Sta77}
Richard Stanley.
\newblock Relative invariants of finite groups generated by pseudo-reflections.
\newblock {\em J. Algebra}, 49:134--148, 1977.

\bibitem{Tan82}
Toshiyuki Tanisaki.
\newblock {Defining ideals of the closures of the conjugaty classes and
  representations of the Weyl groups}.
\newblock {\em Tohoku Mathematical Journal}, 34:575--585, 1982.

\end{thebibliography}
\bibliographystyle{hplain}}
\end{document}